\documentclass[oneside,english,reqno]{amsart}

\usepackage[T1]{fontenc}
\usepackage[latin9]{inputenc}
\usepackage{babel}
\usepackage{amsthm}
\usepackage{amssymb}
\usepackage{graphicx}
\PassOptionsToPackage{normalem}{ulem}
\usepackage{ulem}
\usepackage[unicode=true,pdfusetitle,
 bookmarks=true,bookmarksnumbered=false,bookmarksopen=false,
 breaklinks=false,pdfborder={0 0 1},backref=false,colorlinks=false]
 {hyperref}
\usepackage{breakurl}

\makeatletter

\numberwithin{equation}{section}
\numberwithin{figure}{section}
\newenvironment{lyxlist}[1]
{\begin{list}{}
{\settowidth{\labelwidth}{#1}
 \setlength{\leftmargin}{\labelwidth}
 \addtolength{\leftmargin}{\labelsep}
 }}
{\end{list}}
\theoremstyle{plain}
\newtheorem{thm}{\protect\theoremname}[section]
  \theoremstyle{definition}
  \newtheorem{defn}[thm]{\protect\definitionname}
  \theoremstyle{plain}
  \newtheorem{lem}[thm]{\protect\lemmaname}
  \theoremstyle{plain}
  \newtheorem{algorithm}[thm]{\protect\algorithmname}
  \theoremstyle{remark}
  \newtheorem{rem}[thm]{\protect\remarkname}
  \theoremstyle{plain}
  \newtheorem{prop}[thm]{\protect\propositionname}
  \theoremstyle{plain}
  \newtheorem{cor}[thm]{\protect\corollaryname}
  \theoremstyle{definition}
  \newtheorem{example}[thm]{\protect\examplename}

\newcommand{\lip}{\mbox{\rm lip}\,}

\newcommand{\gph}{\mbox{\rm Graph}}

\newcommand{\pos}{\mbox{\rm pos}}

\newcommand{\conv}{\mbox{\rm conv}}
\newcommand{\spanm}{\mbox{\rm span}}

\makeatother

  \providecommand{\algorithmname}{Algorithm}
  \providecommand{\corollaryname}{Corollary}
  \providecommand{\definitionname}{Definition}
  \providecommand{\examplename}{Example}
  \providecommand{\lemmaname}{Lemma}
  \providecommand{\propositionname}{Proposition}
  \providecommand{\remarkname}{Remark}
\providecommand{\theoremname}{Theorem}

\begin{document}
\title[Algorithm analysis for feasibility problems]{Improved analysis of algorithms based on supporting halfspaces and quadratic programming for the convex intersection and feasibility problems}

\begin{abstract}
This paper improves the algorithms based on supporting halfspaces
and quadratic programming for convex set intersection problems in
our earlier paper in several directions. First, we give conditions
so that much smaller quadratic programs (QPs) and approximate projections
arising from partially solving the QPs are sufficient for multiple-term
superlinear convergence for nonsmooth problems. Second, we identify
additional regularity, which we call the second order supporting hyperplane
property (SOSH), that gives multiple-term quadratic convergence. Third,
we show that these fast convergence results carry over for the convex
inequality problem. Fourth, we show that infeasibility can be detected
in finitely many operations. Lastly, we explain how we can use the
dual active set QP algorithm of Goldfarb and Idnani to get useful
iterates by solving the QPs partially, overcoming the problem of solving
large QPs in our algorithms.
\end{abstract}

\author{C.H. Jeffrey Pang}

\curraddr{Department of Mathematics\\ 
National University of Singapore\\ 
Block S17 08-11\\ 
10 Lower Kent Ridge Road\\ 
Singapore 119076 }

\email{matpchj@nus.edu.sg}


\date{\today{}}

\maketitle
\tableofcontents{}

\section{Introduction}

We consider two different problems in this paper. First, let $K_{1},\dots,K_{r}$
be $r$ closed convex sets in a Hilbert space $X$. The \emph{Set
Intersection Problem }(SIP) is 
\begin{equation}
\mbox{(SIP):}\quad\mbox{Find }x\in K:=\bigcap_{i=1}^{r}K_{i}\mbox{, where }K\neq\emptyset.\label{eq:SIP}
\end{equation}
The \emph{Convex Inequality Problem }(CIP) is 
\begin{equation}
\mbox{(CIP):}\quad\mbox{ For a convex }f:\mathbb{R}^{n}\to\mathbb{R}\mbox{, find }x\in\mathbb{R}^{n}\mbox{ s.t. }f(x)\leq0.\label{eq:CIP}
\end{equation}

This paper improves on the results in \cite{cut_Pang12}, where we
studied convergence results for accelerating convergence of algorithms
for the SIP. The idea there was to collect as many supporting halfspaces
generated by the projection process to create a polyhedron that is
an outer approximation of $K$. Then one can project onto this polyhedron
using quadratic programming. See Figure \ref{fig:alt-proj-compare}.

\begin{figure}[h]
\includegraphics[scale=0.4]{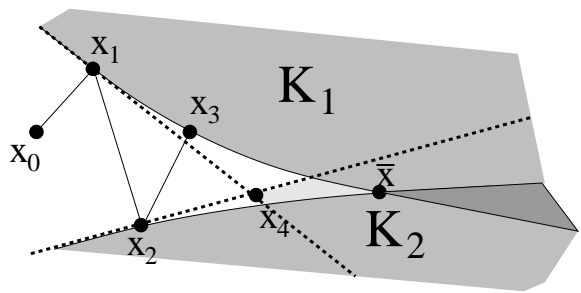}

\caption{\label{fig:alt-proj-compare}The method of alternating projections
on two convex sets $K_{1}$ and $K_{2}$ in $\mathbb{R}^{2}$ with
starting iterate $x_{0}$ arrives at $x_{3}$ in three iterations.
Consider the supporting halfspaces planes of $K_{1}$ and $K_{2}$
at $x_{1}$ and $x_{2}$. The projection of $x_{1}$ onto the intersection
of these halfspaces, which is $x_{4}$, is much closer to the point
$\bar{x}$ than $x_{3}$, especially when the boundary of $K_{1}$
and $K_{2}$ have fewer second order effects and when the angle between
the boundary of $K_{1}$ and $K_{2}$ is small. On the other hand,
the point $x_{3}$ is ruled out by the supporting hyperplane of $K_{2}$
passing through $x_{2}$.}
\end{figure}

We note that the idea of supporting halfspaces and quadratic programming
was studied in \cite{G-P98,G-P01}, but for the CIP when $f(\cdot)$
is the maximum of a finite number of smooth functions. Quadratic programs
with one affine constraint (not necessarily of codimension 1) and
a halfspace were used to accelerate algorithms for the CIP in \cite{Pierra84,BausCombKruk06}.
The idea of using QPs was also present in other works on the CIP (for
example \cite{Fukushima82}).

A popular method for solving the SIP is the method of alternating
projections. We highlight the references \cite{BB96_survey,BR09,Censor84,CensorZenios97,Combettes93,Combettes96,Deutsch95,Deutsch01_survey,EsRa11},
as well as \cite[Chapter 9]{Deustch01} and \cite[Subsubsection 4.5.4]{BZ05},
for an introduction on the SIP and their applications. The papers
\cite{GPR67,GK89,BDHP03} explored acceleration methods for the method
of alternating projections. Another acceleration method is the Dos
Santos method \cite{DosSantos87,DePierro81}, which is based on Cimmino's
method for linear equations \cite{Cimmino38}.

We remark that the treatment for the case when all the sets are affine
spaces are covered in \cite{sub_BAP}.

In this paper, we deal with the case where the $f(\cdot)$ in the
CIP were convex but not smooth. Some early work on (not necessarily
convex) inequality problems are \cite{Robinson_76_CIP,Polak_Mayne79,Mayne_Polak_Heunis81,Fukushima82},
and we elaborate on their contributions in this introduction.  A
related work is \cite{FletcherLeyffer03}, where SQP methods are used
to solve nonconvex but smooth CIPs.

In \cite{Robinson_76_CIP}, Robinson considered the \emph{$K$-Convex
Inequality Problem} (KCIP), which is a generalization of the (CIP).
For $f:\mathbb{R}^{n}\to\mathbb{R}^{m}$, and a closed convex cone
$K\subset\mathbb{R}^{m}$, we write $y_{1}\leq_{K}y_{2}$ if $y_{2}-y_{1}\in K$.
The KCIP is defined by 
\begin{equation}
\mbox{(KCIP):}\mbox{ For }f:\mathbb{R}^{n}\to\mathbb{R}^{m}\mbox{ and }C\subset\mathbb{R}^{n}\mbox{, find }x\in C\mbox{ s.t. }f(x)\leq_{K}0.\label{eq:KCIP}
\end{equation}

Robinson's algorithm in \cite{Robinson_76_CIP} for the CIP can be
described as follows: At each iterate $x_{i}$, a subgradient $y_{i}\in\partial f(x_{i})$
is obtained, and the halfspace 
\begin{equation}
H_{i}^{\leq}:=\{x\in\mathbb{R}^{n}\mid f(x_{i})+\left\langle y_{i},x-x_{i}\right\rangle \leq0\}\label{eq:halfspace-type}
\end{equation}
 contains $f^{-1}((-\infty,0])$. The next iterate $x_{i+1}$ is obtained
by projecting $x_{i}$ onto $H_{i}^{\leq}$. Assuming regularity and
convexity (and no smoothness), Robinson proved that the CIP (or more
generally, the KCIP) converges at least linearly. 

The idea of collecting halfspaces and projecting onto their intersection
using quadratic programming can be carried over for the CIP. In the
CIP, the halfspaces are of type \eqref{eq:halfspace-type}.

A related paper on the CIP is \cite{Fukushima82}, where Fukushima
obtained finite convergence for the CIP when $f^{-1}((-\infty,0])$
has nonempty interior, assuming only convexity and not smoothness.
The idea is to try to find an $x$ satisfying 
\[
f(x)\leq-\epsilon_{k}
\]
at iteration $k$, where $\{\epsilon_{k}\}$ is a sequence of positive
numbers converging to zero at a rate slower than any linearly converging
sequence. (It appears that \cite{DePierroIusem88} have come up with
a similar result independently.) For smooth problems, this idea can
be traced back to \cite{Polak_Mayne79,Mayne_Polak_Heunis81} or possibly
earlier. 

Other papers on the SIP are \cite{Kiwiel95}, where the interest is
on problems where $r$, the number of closed convex sets $K_{i}$,
is large. A method for the best approximation problem (stated in \eqref{eq:Proj-pblm})
is Dykstra's algorithm \cite{Dykstra83,BD86,Han88}. There has been
recent interest in nonconvex SIP problems \cite{LewisMalick08,LLM09_lin_conv_alt_proj}.
We believe that an adaptation of our algorithm can be useful for nonconvex
problems.

It appears that prevailing algorithms for the SIP (see for example
the algorithms in \cite{CensorChenCombettesDavidiHerman12,EsRa11})
do not exploit smoothness of the sets and fall back to a Newton-like
method and achieve superlinear convergence in the manner of \cite{G-P98,G-P01}
and the algorithms of this paper. We note however that variants of
the algebraic reconstruction technique (ART), which try to find a
point in the intersection of hyperslabs rather than general convex
sets, can achieve finite convergence (See \cite{HermanChen08} and
the references therein).

\subsection{\label{sub:Contrib}Contributions of this paper}

This paper improves on the algorithm for the SIP in \cite{cut_Pang12}.
In \cite{cut_Pang12}, a multiple-term superlinearly convergent algorithm
for nonsmooth SIPs for the case when $X=\mathbb{R}^{n}$ was proposed,
but the algorithm there requires one to solve impractically huge QPs.
In this paper, we show that the following adjustments, reflected in
Algorithm \ref{alg:Mass-proj-alg}, maintain such fast convergence:
\begin{itemize}
\item Instead of accumulating $r\bar{p}$ halfspaces (where $\bar{p}$ is
a huge parameter) as proposed in \cite{cut_Pang12}, superlinear convergence
can be achieved if the normals of two of the halfspaces produced by
the projection process are close enough to each other. (See Theorem
\ref{thm:conv-SIP}.) This condition is weaker and easier to check
in practice, and can greatly reduce the size of the QPs that we need
to solve to maintain superlinear convergence. (See Remark \ref{rem:small-S-i}.)
We present Corollaries \ref{cor:p-term-suplin} and \ref{cor:fast-sip}
based on this result. In particular, Corollary \ref{cor:fast-sip}
states that when the boundaries of the sets are smooth, our algorithm
reduces to a Newton method, and our framework can prove that the convergence
is indeed superlinear or quadratic.
\item The requirement of projecting onto the intersection of the halfspaces
is relaxed, reflecting that an approximate projection can also guarantee
superlinear convergence. The quadratic programming subproblem for
projecting onto the polyhedron can be solved partially using a dual
active set algorithm of \cite{Goldfarb_Idnani}, then extrapolated
to a feasible point in the polyhedron. See Section \ref{sec:inner-GI}
and Remark \ref{rem:step-2-relax}.
\end{itemize}
Large parts of the proofs in \cite{cut_Pang12} for this result remain
the same here. We only focus on the additional details without repeating
the proofs that are largely unchanged from \cite{cut_Pang12}.

We introduce the \emph{Second Order Supporting Hyperplane} (SOSH)
property (Definition \ref{def:SOSH}), which we show is present in
sets defined by convex inequalities (Proposition \ref{prop:C2-SOSH}).
Moreover, the SOSH property is preserved under intersections under
a constraint qualification (Proposition \ref{prop:SOSH-intersect}).
Under the SOSH property, we can achieve multiple-term quadratic convergence
of the SIP algorithm.

We then show that a multiple-term superlinear convergence for a nonsmooth
CIP \eqref{eq:CIP} can be achieved using the techniques studied for
the SIP in Section \ref{sec:CIP}. 

Next, we look at infeasibility detection. In \cite[Section 6]{cut_Pang12},
we had discussed infeasibility detection for the Best Approximation
Problem (BAP) (See \eqref{eq:Proj-pblm}). In Theorem \ref{thm:No-cluster-pt},
we show that under reasonable conditions, algorithms for the SIP,
CIP and BAP do not have strong cluster points in the infeasible case
in finite dimensions, and can even obtain a certificate of infeasibility
in a finite number of operations. We make further observations about
the BAP in Theorem \ref{thm:BAP-to-infinity}.

Lastly, we explain that the dual active set QP algorithm of Goldfarb
and Idnani \cite{Goldfarb_Idnani} gives good iterates after each
inner iteration even when the QPs in our algorithms are not solved
fully. Such a property is useful since the QPs that we solve may be
large and difficult to solve to optimality.

\subsection{Notation}

We recall some standard notation in convex analysis that are helpful
for the rest of the paper. As our results rely on the compactness
of the unit sphere, we only treat the finite dimensional case here.
Let $C\subset\mathbb{R}^{n}$ be a closed convex set, $f:\mathbb{R}^{n}\to\mathbb{R}$
be a convex function, and $x\in C$. Then we have the following notation:
\begin{lyxlist}{00.00.0000}
\item [{$N_{C}(x)$}] The \emph{normal cone} $N_{C}(x)$ \emph{at the point
$x\in C$} is the set \\
$\{v\mid\langle v,y-x\rangle\leq0\mbox{ for all }y\in C\}$.
\item [{$\mathbb{B}(x,r)$}] Ball with center $x$ and radius $r$: $\mathbb{B}(x,r):=\{y\mid\|y-x\|\leq r\}$.\\
We write $\mathbb{B}:=\mathbb{B}(0,1)$.
\item [{$\partial f(x)$}] The \emph{subdifferential} \emph{of $f$ at
$x$}: \\
$\partial f(x):=\{y\mid f(x^{\prime})\geq f(x)+\left\langle y,x^{\prime}-x\right\rangle $
for all $x^{\prime}\in\mathbb{R}^{n}\}$.
\item [{$d(x,S)$}] The distance of $x$ to a set $S\subset\mathbb{R}^{n}$:
$d(x,S):=\inf_{s\in S}\|x-s\|$.
\item [{$\lip f(x)$}] The \emph{Lipschitz modulus of $f$ at $x$}: $\lip f(x):=\limsup_{{x_{1},x_{2}\to x\atop x_{1}\neq x_{2}}}\frac{\|f(x_{1})-f(x_{2})\|}{\|x_{1}-x_{2}\|}$.
\item [{$\pos(S)$}] For a set $S\subset\mathbb{R}^{n}$, the \emph{positive
hull} is the set \\
$\pos(S):=\{ts\mid t\in[0,\infty),s\in S\}$.
\item [{$R(S)$}] For a convex set $S\subset\mathbb{R}^{n}$, the \emph{recession
cone }is the set\\
$\{d:x+td\in S$ for all $t\geq0$ and $x\in S\}$.
\end{lyxlist}
We denote $F:\mathbb{R}^{n}\rightrightarrows\mathbb{R}^{m}$ to be
a \emph{set-valued map} that maps a point in $\mathbb{R}^{n}$ to
a subset of $\mathbb{R}^{m}$. A set-valued map $F:\mathbb{R}^{n}\rightrightarrows\mathbb{R}^{m}$
is \emph{outer semicontinuous }if its \emph{graph }$\gph(F):=\{(x,y)\mid y\in F(x)\}$
is closed. A convex cone $C\subset\mathbb{R}^{n}$ is \emph{pointed
}if it does not contain a line. The notation ``$\partial$'' can
also mean the boundary of a closed set, which should not lead to confusion
with the subdifferential. In our proofs, we also make use of the Pompieu
Hausdorff distance. We refer the reader to standard texts in convex
and variational analysis \cite{Rockafellar70,HiriartUrrutyLamerechal93a,RW98,Cla83,Mor06}
for more information.

\section{Preliminary results}

In this section, we collect a few results that are nonstandard, but
will be useful for the rest of the paper.
\begin{defn}
(Fej\'{e}r monotone sequence) Let $X$ be a Hilbert space, $C\subset X$
be a closed convex set, and $\{x_{i}\}$ be a sequence in $X$. We
say that $\{x_{i}\}$ is\emph{ Fej\'{e}r monotone with respect to
$C$} if 
\[
\|x_{i+1}-c\|\leq\|x_{i}-c\|\mbox{ for all }c\in C\mbox{ and }i=1,2,\dots
\]

\end{defn}
A tool for obtaining a Fej\'{e}r monotone sequence is stated below.
\begin{thm}
\label{thm:Fejer-contraction}(Fej\'{e}r attraction property) Let
$X$ be a Hilbert space. For a closed convex set $C\subset X$, $x\in X$,
$\lambda\in[0,2]$, and the projection $P_{C}(x)$ of $x$ onto $C$,
let the \emph{relaxation operator }$R_{C,\lambda}:X\to X$ \cite{Agmon54}
be defined by 
\[
R_{C,\lambda}(x)=x+\lambda(P_{C}(x)-x).
\]
Then 
\begin{equation}
\|R_{C,\lambda}(x)-c\|^{2}\leq\|x-c\|^{2}-\lambda(2-\lambda)d(x,C)^{2}\mbox{ for all }y\in C.\label{eq:Fejer-mon-eq}
\end{equation}

\end{thm}
Here are some consequences of Fej\'{e}r monotonicity. We take our
results from \cite[Theorem 4.5.10 and Lemma 4.5.8]{BZ05}.
\begin{thm}
\label{thm:Fejer-ppty}(Properties of Fej\'{e}r monotonicity) Let
$X$ be a Hilbert space, let $C\subset X$ be a closed convex set
and let $\{x_{i}\}$ be a Fej\'{e}r monotone sequence with respect
to $C$. Then 
\begin{enumerate}
\item $\{x_{i}\}$ is bounded and $d(C,x_{i+1})\leq d(C,x_{i})$, and
\item $\{x_{i}\}$ has at most one weak cluster point in $C$.
\end{enumerate}
\end{thm}
The following result is elementary and proved in \cite{cut_Pang12},
and will also be used in the proof of Theorem \ref{thm:superlin-CIP}.
Recall that the dual cone $K^{+}\subset\mathbb{R}^{n}$ of a convex
cone $K\subset\mathbb{R}^{n}$ is 
\[
K^{+}:=\{y\mid\left\langle x,y\right\rangle \geq0\mbox{ for all }x\in K\}.
\]

\begin{lem}
\label{lem:pointed-cone-modulus}(Pointed cone) For a closed pointed
convex cone $K\subset\mathbb{R}^{n}$, there is a unit vector $d$
in $K^{+}$, the dual cone of $K$, and some $c>0$ such that $\mathbb{B}(d,c)\subset K^{+}$.
For any unit vector $v\in K$, we have $d^{T}v\geq c$.

Moreover, suppose $\lambda_{i}\geq0$, and $v_{i}$ are unit vectors
in $K$ for all $i$, and $\sum_{i=1}^{\infty}\lambda_{i}v_{i}$ converges
to $\bar{v}$. Clearly, $\bar{v}\in K$. Then $\|\sum_{i=1}^{\infty}\lambda_{i}v_{i}\|\geq c\sum_{i=1}^{\infty}\lambda_{i}$,
which also implies that $\sum_{i=1}^{\infty}\lambda_{i}$ is finite.
\end{lem}
We now recall the definition of semismoothness.
\begin{defn}
\cite{Mifflin77} (Semismoothness) Let $\Phi:\mathbb{R}^{n}\to\mathbb{R}$
be convex. We say that $\Phi$ is \emph{semismooth at $x$} if it
is directionally differentiable at $x$ and for any $V\in\partial\Phi(x+h)$,
\[
\Phi(x+h)-\Phi(x)-Vh=o(\|h\|).
\]
We say that $\Phi$ is \emph{strongly semismooth at $x$ }if $\Phi$
is semismooth at $x$ and 
\[
\Phi(x+h)-\Phi(x)-Vh=O(\|h\|^{2}).
\]

\end{defn}
Semismoothness is also defined for vector-valued functions that need
not be convex, but this definition above is enough for our purposes.
Moreover, it is proved that convexity implies semismoothness, so semismoothness
is superfluous in our context. But we shall need to use strong semismoothness
later for Theorem \ref{thm:superlin-CIP}(b).

\section{\label{sec:improve-shqp}Improving convergence results of the SIP}

In this section, we show how to improve the convergence results for
the SIP in \cite{cut_Pang12} as detailed in Subsection \ref{sub:Contrib}.
We recall an adaptation of the $\bar{p}$-term superlinear convergent
algorithm of \cite{cut_Pang12} for the SIP.
\begin{algorithm}
\label{alg:Mass-proj-alg}(SHQP algorithm for the SIP) For a starting
iterate $x_{0}\in\mathbb{R}^{n}$ and closed convex sets $K_{l}\subset\mathbb{R}^{n}$,
where $1\leq l\leq r$, find a point in $K:=\cap_{l=1}^{r}K_{l}$. 

\textbf{\uline{Step 0}}\uline{:}\textbf{ }Set $i=0$, and let
$\bar{p}$ be a positive integer.

\textbf{\uline{Step 1:}} For $l\in\{1,\dots,r\}$, define $x_{i}^{(l)}\in\mathbb{R}^{n}$,
$a_{i}^{(l)}\in\mathbb{R}^{n}$ and $b_{i}^{(l)}\in\mathbb{R}$ by
\begin{eqnarray*}
x_{i}^{(l)} & = & P_{K_{l}}(x_{i}),\\
a_{i}^{(l)} & = & x_{i}-x_{i}^{(l)},\\
\mbox{and }b_{i}^{(l)} & = & \left\langle a_{i}^{(l)},x_{i}^{(l)}\right\rangle .
\end{eqnarray*}
For each $i$, we let $l_{i}^{*}\in\{1,\dots,r\}$ be such that 
\[
l_{i}^{*}:=\arg\max_{1\leq l\leq r}\|x_{i}-P_{K_{l}}(x_{i})\|.
\]

\textbf{\uline{Step 2:}} Choose $S_{i}\subset\{\max(i-\bar{p},0),\dots,i\}\times\{1,\dots,r\}$,
and define $\tilde{F}_{i}\subset\mathbb{R}^{n}$ by
\begin{eqnarray}
\tilde{F}_{i} & := & \bigcap_{(j,l)\in S_{i}}H_{(j,l)}^{\leq},\label{eq:tilde-F-i}\\
\mbox{where }H_{(j,l)}^{\leq} & := & \left\{ x:\left\langle a_{j}^{(l)},x\right\rangle \leq b_{j}^{(l)}\right\} .
\end{eqnarray}

In other words, $H_{(j,l)}^{\leq}$ is the halfspace generated by
projecting $x_{j}$ onto $K_{l}$. Let $x_{i+1}$ be chosen such that 

(1) $x_{i+1}\in\tilde{F}_{i}$,

(2) (Fej\'{e}r attraction) $\|x_{i+1}-c\|\leq\|x_{i}-c\|$ for all
$c\in K$, and 

(3) $x_{i}-x_{i+1}$ lies in $\conv(\{a_{j}^{(l)}:(j,l)\in S_{i}\})$.

\textbf{\uline{Step 3:}}\textbf{ }Set $i\leftarrow i+1$, and go
back to step 1.
\end{algorithm}
We try to keep our notation consistent with that of \cite{cut_Pang12},
but we decided that it is better to use the index $i$ in a different
manner from \cite{cut_Pang12}.
\begin{rem}
\label{rem:small-S-i}(Choice of $S_{i}$) We leave the choice of
$S_{i}$ open\textbf{ }in Algorithm \ref{alg:Mass-proj-alg}. The
choice $S_{i}=\{\max(i-\bar{p},0),\dots,i\}\times\{1,\dots,r\}$ was
studied in \cite{cut_Pang12}. We will also look at the choice $S_{i}=\{i\}\times\{1,\dots,r\}$
in Corollary \ref{cor:fast-sip}. 
\begin{rem}
\label{rem:step-2-relax}(Step 2 of Algorithm \ref{alg:Mass-proj-alg})
One way to obtain $x_{i+1}$ is by projecting $x_{i}$ onto $\tilde{F}_{i}$.
Such an $x_{i+1}$ would satisfy the conditions (1), (2) and (3) of
step 2 by the properties of the projection. The argument to see how
(2) is satisfied is simple: The polyhedron $\tilde{F}_{i}$ contains
$K$, and by the Fej\'{e}r attractive property of projections, $\|x_{i+1}-c\|\leq\|x_{i}-c\|$
for all $c\in\tilde{F}_{i}$. Condition (3) follows from the KKT conditions
of the projection operation.

A point satisfying the conditions (1), (2) and (3) may be easier to
obtain than the projection. For example, one can use the dual active
set quadratic programming algorithm of Goldfarb and Idnani \cite{Goldfarb_Idnani}
to obtain a point $\tilde{x}_{i+1}$ that is the projection of $x_{i}$
onto the polyhedron formed by intersecting a subset of $\{H_{(j,l)}^{\leq}\}_{(j,l)\in S_{i}}$.
If the point $z:=\lambda[\tilde{x}_{i+1}-x_{i}]+x_{i}$ for some $\lambda\in[1,2]$
is such that $z\in\tilde{F}_{i}$, then in view of Theorem \ref{thm:Fejer-contraction},
$x_{i+1}$ can be taken to be $z$. For more details on applying the
dual quadratic programming algorithm to solve the SIP, we refer to
Section \ref{sec:inner-GI}.
\end{rem}
\end{rem}
Before we remark on the $\bar{p}$-term quadratic convergence of the
algorithm in \cite{cut_Pang12}, we need to look at a theorem on convex
sets proved in \cite{cut_Pang12} and the SOSH property defined and
studied afterward.
\begin{thm}
\cite{cut_Pang12}\label{thm:radiality}(Supporting hyperplane near
a point) Suppose $C\subset\mathbb{R}^{n}$ is a closed convex set,
and let $\bar{x}\in C$. Then for any $\epsilon>0$, there is a $\delta>0$
such that for any point $x\in[\mathbb{B}(\bar{x},\delta)\cap C]\backslash\{\bar{x}\}$
and supporting hyperplane $A$ of $C$ with unit normal $v\in N_{C}(x)$
at the point $x$, we have $\frac{d(\bar{x},A)}{\|x-\bar{x}\|}\leq\epsilon$.

Since $d(\bar{x},A)=-\left\langle v,\bar{x}-x\right\rangle $, the
conclusion can be replaced by 
\begin{equation}
0\leq-\left\langle v,\bar{x}-x\right\rangle \leq\epsilon\|\bar{x}-x\|.\label{eq:little-SOSH}
\end{equation}
\end{thm}
\begin{defn}
\label{def:SOSH}(Second order supporting hyperplane property) Suppose
$C\subset\mathbb{R}^{n}$ is a closed convex set, and let $\bar{x}\in C$.
We say that $C$ has the \emph{second order supporting hyperplane
(SOSH) property at $\bar{x}$ }(or more simply, $C$ is SOSH at $\bar{x}$)
if there are $\delta>0$ and $M>0$ such that for any point $x\in[\mathbb{B}_{\delta}(\bar{x})\cap C]\backslash\{\bar{x}\}$
and $v\in N_{C}(x)$ such that $\|v\|=1$, we have 
\begin{equation}
0\leq-\left\langle v,\bar{x}-x\right\rangle \leq M\|\bar{x}-x\|^{2}.\label{eq:SOSH-1}
\end{equation}

\end{defn}
It is clear how \eqref{eq:little-SOSH} compares with \eqref{eq:SOSH-1}.
The next two results show that SOSH is prevalent in applications.
\begin{prop}
\label{prop:C2-SOSH}(Smoothness implies SOSH) Suppose a convex function
$f:\mathbb{R}^{n}\to\mathbb{R}$ is $\mathcal{C}^{2}$ at $\bar{x}$.
Then the set $C=\{x\mid f(x)\leq0\}$ is SOSH at $\bar{x}$.\end{prop}
\begin{proof}
Consider $\bar{x},x\in C$. In order for the problem to be meaningful,
we shall only consider the case where $f(\bar{x})=0$. We also assume
that $f(x)=0$ so that $C$ has a supporting hyperplane at $x$. An
easy calculation gives $N_{C}(\bar{x})=\mathbb{R}_{+}\{\nabla f(\bar{x})\}$
and $N_{C}(x)=\mathbb{R}_{+}\{\nabla f(x)\}$. Convexity ensures that
$0\leq\frac{-\nabla f(x)(\bar{x}-x)}{\|x-\bar{x}\|^{2}}$ by Theorem
\ref{thm:radiality}.

Without loss of generality, let $\bar{x}=0$. We have 
\[
f(x)=f(0)+\nabla f(0)x+\frac{1}{2}x^{T}\nabla^{2}f(0)x+o(\|x\|^{2})
\]
Since $f(x)=f(0)=0$ and $[\nabla f(0)-\nabla f(x)]x=x^{T}\nabla^{2}f(0)x+o(\|x\|^{2})$,
we have 
\[
-\nabla f(x)(x)=[\nabla f(0)-\nabla f(x)]x+\frac{1}{2}x^{T}\nabla^{2}f(0)x+o(\|x\|^{2})=O(\|x\|^{2}).
\]
 Therefore, we are done. \end{proof}
\begin{prop}
\label{prop:SOSH-intersect}(SOSH under intersection) Suppose $K_{l}\subset\mathbb{R}^{n}$
are closed convex sets that are SOSH at $\bar{x}$ for $l\in\{1,\dots,r\}$.
Let $K:=\cap_{l=1}^{r}K_{l}$, and suppose that 
\begin{equation}
\sum_{l=1}^{r}v_{l}=0,\, v_{l}\in N_{K_{l}}(\bar{x})\mbox{ implies }v_{l}=0\mbox{ for all }l\in\{1,\dots,r\}.\label{eq:CQ-bar-x}
\end{equation}
Then $K$ is SOSH at $\bar{x}$.\end{prop}
\begin{proof}
Since each $K_{l}$ is SOSH at $\bar{x}$, we can find $\delta>0$
and $M>0$ such that for all $l\in\{1,\dots,r\}$ and $x\in K_{l}\cap\mathbb{B}_{\delta}(\bar{x})$
and $v\in N_{K_{l}}(x)$, we have 
\[
0\leq-\left\langle v,\bar{x}-x\right\rangle \leq M\|v\|\|\bar{x}-x\|^{2}.
\]

\textbf{\uline{Claim 1: We can reduce $\delta>0$ if necessary
so that}}
\begin{eqnarray}
 &  & \sum_{l=1}^{r}v_{l}=0,\, v_{l}\in N_{K_{l}}(x)\label{eq:CQ-all-x}\\
 &  & \qquad\qquad\mbox{ implies }v_{l}=0\mbox{ for all }l\in\{1,\dots r\}\mbox{ and }x\in K\cap\mathbb{B}_{\delta}(\bar{x}).\nonumber 
\end{eqnarray}
Suppose otherwise. Then we can find $\{x_{i}\}_{i=1}^{\infty}\in K$
such that $\lim x_{i}=\bar{x}$ and for all $i>0$, there exists $v_{l,i}\in N_{K_{l}}(x_{i})$
such that $\sum_{l=1}^{r}v_{l,i}=0$ but not all $v_{l,i}=0$. We
can normalize so that $\|v_{l,i}\|\leq1$, and for each $i$, $\max_{l}\|v_{l,i}\|=1$.
By taking a subsequence if necessary, we can assume that $\lim v_{l,i}$,
say $\bar{v}_{l}$, exists for all $l$. Not all $\bar{v}_{l}$ can
be zero, but $\sum_{l=1}^{r}\bar{v}_{l}=0$. The outer semicontinuity
of the normal cone mapping implies that $\bar{v}_{l}\in N_{K_{l}}(\bar{x})$.
This is a contradiction to \eqref{eq:CQ-bar-x}, which ends the proof
of Claim 1.

\textbf{\uline{Claim 2: There exists a constant $M^{\prime}$ such
that whenever $x\in\mathbb{B}_{\delta}(\bar{x})\cap K$, $v_{l}\in N_{K_{l}}(x)$
and $v=\sum_{l=1}^{r}v_{l}$, then $\max\|v_{l}\|\leq M^{\prime}\|v\|$. }}

Suppose otherwise. Then for each $i$, there exists $x_{i}\in\mathbb{B}_{\delta}(\bar{x})\cap K$
and $\tilde{v}_{l,i}\in N_{K_{l}}(x_{i})$ such that $\tilde{v}_{i}=\sum_{l=1}^{r}\tilde{v}_{l,i}$,
$\|\tilde{v}_{i}\|\leq\frac{1}{i}$, and $\max_{l}\|\tilde{v}_{l,i}\|=1$
for all $i$. As we take limits to infinity, this would imply that
\eqref{eq:CQ-all-x} is violated, a contradiction. This ends the proof
of Claim 2.

Since \eqref{eq:CQ-all-x} is satisfied, this means that $N_{K}(x)=\sum_{l=1}^{r}N_{K_{l}}(x)$
for all $x\in\mathbb{B}_{\delta}(\bar{x})\cap K$ by the intersection
rule for normal cones in \cite[Theorem 6.42]{RW98}. Then each $v\in N_{K}(x)$
can be written as a sum of elements in $N_{K_{l}}(x)$, say $v=\sum_{l=1}^{r}v_{l}$,
where $v_{l}\in N_{K_{l}}(x)$, and $\max\|v_{l}\|\leq M^{\prime}\|v\|$.
Then 
\begin{eqnarray*}
-\left\langle v,\bar{x}-x\right\rangle  & = & \sum_{l=1}^{r}-\left\langle v_{l},\bar{x}-x\right\rangle \\
 & \leq & M\|\bar{x}-x\|^{2}\sum_{l=1}^{r}\|v_{l}\|\quad\leq\quad M\|\bar{x}-x\|^{2}rM^{\prime}\|v\|.
\end{eqnarray*}
Thus we are done.
\end{proof}
 We recall the definition of local metric inequality, sometimes
referred to as linear regularity.
\begin{defn}
\label{def:Loc-lin-reg}(Local metric inequality) We say that a collection
of closed sets $K_{l}$, $l=1,\dots,r$ satisfies the \emph{local
metric inequality }at $\bar{x}$ if there are $\bar{\kappa}>0$ and
$\delta>0$ such that
\begin{equation}
d(x,\cap_{l=1}^{r}K_{l})\leq\bar{\kappa}\max_{1\leq l\leq r}d(x,K_{l})\mbox{ for all }x\in\mathbb{B}_{\delta}(\bar{x}).\label{eq:loc-metric-ineq}
\end{equation}

\end{defn}
The following result is well-known, and we haven't been able to pin
an original source. We refer to \cite{cut_Pang12} for a discussion
on its proof.
\begin{lem}
\label{lem:loc-metr-ineq-condn}(Condition for local metric inequality)
Suppose $\bar{x}\in K$, $K_{l}\subset\mathbb{R}^{n}$ are closed
convex sets for $l=1,\dots,r$ and $K:=\cap_{l=1}^{r}K_{l}$. Suppose
that 
\begin{enumerate}
\item If $\sum_{l=1}^{r}v_{l}=0$ for some $v_{l}\in N_{K_{l}}(\bar{x})$,
then $v_{l}=0$ for all $l=1,\dots,r$.
\end{enumerate}
Then $\{K_{l}\}_{l=1}^{r}$ satisfies the local metric inequality
at $\bar{x}$ for some $\bar{\kappa}\geq0$.
\end{lem}
We now prove our convergence result for Algorithm \ref{alg:Mass-proj-alg}.
\begin{thm}
\label{thm:conv-SIP}(Convergence rates for the SIP) Suppose Algorithm
\ref{alg:Mass-proj-alg} with parameter $\bar{p}$ produces a sequence
$\{x_{i}\}$ that converges to a point $\bar{x}\in K$, and the convergence
is not finite. Suppose also that 
\begin{enumerate}
\item If $\sum_{l=1}^{r}v_{l}=0$ for some $v_{l}\in N_{K_{l}}(\bar{x})$,
then $v_{l}=0$ for all $l=1,\dots,r$.
\end{enumerate}
In view of condition (1) and Lemma \ref{lem:loc-metr-ineq-condn},
$\{K_{l}\}_{l=1}^{r}$ satisfies the local metric inequality at $\bar{x}$
with some constant, say $\bar{\kappa}$. Let $\bar{\alpha}$ be such
that 
\[
\bar{\alpha}<\sin^{-1}(1/\bar{\kappa}),
\]
and suppose that the parameter $\bar{p}$ in Algorithm \ref{alg:Mass-proj-alg}
is sufficiently large so that for any $\bar{p}$ unit vectors in $\mathbb{R}^{n}$,
there are two vectors such that the angle between them is at most
$\bar{\alpha}$. 

(a) For any $\epsilon>0$, there is an $I>0$ such that if $I<i<k$
and $\angle a_{j}^{(l)}0a_{k}^{(l_{k}^{*})}\leq\bar{\alpha}$ for
some $(j,l)\in S_{k-1}$, then 
\[
\frac{\|x_{k}-\bar{x}\|}{\|x_{i}-\bar{x}\|}\leq\frac{\|x_{k}-\bar{x}\|}{\|x_{j}-\bar{x}\|}\leq\epsilon.
\]

(b) Suppose in addition all the sets $K_{l}$ have the SOSH property
at $\bar{x}$. Then there are $M>0$ and $I>0$ such that if $I<i<k$
and $\angle a_{j}^{(l)}0a_{k}^{(l_{k}^{*})}\leq\bar{\alpha}$ for
some  $(j,l)\in S_{k-1}$, then 
\begin{equation}
\frac{\|x_{k}-\bar{x}\|}{\|x_{i}-\bar{x}\|^{2}}\leq\frac{\|x_{k}-\bar{x}\|}{\|x_{j}-\bar{x}\|^{2}}<M.\label{eq:p-term-quad-1}
\end{equation}
\end{thm}
\begin{proof}
We break up into two steps:

\textbf{\uline{Step 1:}}\textbf{ Summary of results largely unchanged
from \cite{cut_Pang12}. }

We summarize the results proved for the algorithm in \cite{cut_Pang12}
that still hold for Algorithm \ref{alg:Mass-proj-alg} with minor
modifications.

Since condition (1) holds, there is a constant $\bar{\kappa}$ satisfying
the local metric inequality \eqref{eq:loc-metric-ineq}. It was proved
in \cite{cut_Pang12} that if condition (1) is satisfied, then 
\begin{equation}
\lim_{i\to\infty}d\left(\frac{x_{i}-\bar{x}}{\|x_{i}-\bar{x}\|},N_{K}(\bar{x})\right)=0.\label{eq:conv-to-normal}
\end{equation}
The conclusion \eqref{eq:conv-to-normal} still holds true for Algorithm
\ref{alg:Mass-proj-alg} with exactly the same proof, using only the
weaker requirements of property (3) of step 2 of Algorithm \ref{alg:Mass-proj-alg}
and the outer semicontinuity of the normal mapping of a convex set.
The formula \eqref{eq:conv-to-normal} is the most tedious result
in \cite{cut_Pang12}. With the same steps as presented in \cite{cut_Pang12},
we can use \eqref{eq:conv-to-normal} to prove that $\lim_{i\to\infty}\frac{d(x_{i},K)}{\|x_{i}-\bar{x}\|}=1.$
In view of the local metric inequality \eqref{eq:loc-metric-ineq},
we can consider any $\kappa>\bar{\kappa}$ and get 
\begin{equation}
\|x_{i}-\bar{x}\|\leq\kappa\max_{1\leq l\leq r}d(x_{i},K_{l})\mbox{ for all }i\mbox{ large enough.}\label{eq:approach-ineq}
\end{equation}

\textbf{\uline{Step 2:}}\textbf{ Obtaining conclusions}

Consider the two dimensional affine space $\bar{x}+\spanm\{a_{j}^{(l)},a_{k}^{(l_{k}^{*})}\}$.
Let $x_{k}^{+}:=P_{K_{l_{k}^{*}}}(x_{k})$. Let the projection of
$x_{k}$ and $x_{k}^{+}$ onto this subspace be $\Pi x_{k}$ and $\Pi x_{k}^{+}$.
Let $x_{k}^{++}$ be the projection of $x_{k}$ onto the hyperplane
passing through $\bar{x}$ whose normal is $a_{k}^{(l_{k}^{*})}$,
and let $\Pi x_{k}^{++}$ be similarly defined. The points are indicated
in Figure \ref{fig:geom1}.

\begin{figure}[h]
\includegraphics[scale=0.5]{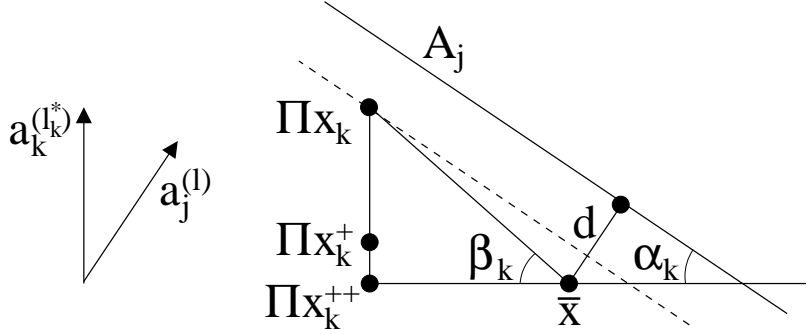}

\caption{\label{fig:geom1}We illustrate the points defined in step 2 of the
proof of Theorem \ref{thm:conv-SIP}. The directions (and not magnitudes)
of $a_{j}^{(l)}$ and $a_{k}^{(l_{k}^{*})}$ (which are the normal
vectors of the halfspaces obtained from projecting $x_{j}$ onto $K_{l}$
and $x_{k}$ onto $K_{l_{k}^{*}}$ respectively) are indicated on
the left. The distance $d$ equals $d(\bar{x},A_{j})$. By pulling
the hyperplane $A_{j}$ towards $\bar{x}$ till it hits $\Pi x_{k}$
(the dashed line), we can prove inequality \eqref{eq:ineq-C}.}
\end{figure}

It is clear that $\|x_{k}-x_{k}^{+}\|=\|x_{k}-P_{K_{l_{k}^{*}}}(x_{k})\|=d(x_{k},K_{l_{k}^{*}})$,
so from \eqref{eq:approach-ineq}, we have, for all $k$ large enough,
\begin{eqnarray}
\|\Pi x_{k}-\bar{x}\| & \leq & \|x_{k}-\bar{x}\|\nonumber \\
 & \leq & \kappa\|x_{k}-P_{K_{l_{k}^{*}}}(x_{k})\|\nonumber \\
 & = & \kappa\|\Pi x_{k}-\Pi x_{k}^{+}\|\nonumber \\
 & \leq & \kappa\|\Pi x_{k}-\Pi x_{k}^{++}\|.\label{eq:ineq-A}
\end{eqnarray}
Hence the angle $\angle(\Pi x_{k})\bar{x}(\Pi x_{k}^{++})$, which
is marked as $\beta_{k}$ in Figure \ref{fig:geom1}, satisfies $\beta_{k}=\sin^{-1}(1/\kappa)$.
Let $\bar{\beta}:=\liminf_{k\to\infty}\beta_{k}$. We must have $\bar{\beta}\geq\sin^{-1}(1/\bar{\kappa})$.
The angle $\alpha_{k}$ marked on Figure \ref{fig:geom1} equals $\angle a_{j}^{(l)}0a_{k}^{(l_{k}^{*})}$,
and satisfies $\alpha_{k}\leq\bar{\alpha}$ by the assumptions in
this result. The $\kappa$ can be chosen such that $\sin\bar{\alpha}<\frac{1}{\kappa}<\frac{1}{\bar{\kappa}}$
so that $\bar{\alpha}<\bar{\beta}$.

Now, let $A_{j}$ be the hyperplane produced by projecting $x_{j}$
onto $K_{l}$ (i.e., $A_{j}=\partial H_{(j,l)}^{\leq}$, the boundary
of $H_{(j,l)}^{\leq}$), and let $d_{j}$ be the distance $d(\bar{x},A_{j})$.
In view of Theorem \ref{thm:radiality}, for any $\epsilon>0$, we
can find $I$ large enough such that 
\begin{equation}
d(\bar{x},A_{j})\leq\epsilon\|P_{K_{l}}(x_{j})-\bar{x}\|\leq\epsilon\|x_{j}-\bar{x}\|\mbox{ for all }j>I.\label{eq:ineq-B}
\end{equation}
A simple argument in plane geometry elaborated in Figure \ref{fig:geom1}
gives, for all $j>I$, 
\begin{equation}
\|\Pi x_{k}-\Pi x_{k}^{++}\|\leq\frac{d(\bar{x},A_{j})}{\sin^{-1}(\beta_{k}-\alpha_{k})}\leq\frac{\epsilon\|x_{j}-\bar{x}\|}{\sin^{-1}(\beta_{k}-\alpha_{k})}.\label{eq:ineq-C}
\end{equation}
Combining inequalities \eqref{eq:ineq-A} and \eqref{eq:ineq-C},
we have, for $k>j>i>I$, 
\begin{eqnarray*}
\|x_{k}-\bar{x}\| & \leq & \kappa\|\Pi x_{k}-\Pi x_{k}^{++}\|\\
 & \leq & \kappa\frac{\epsilon\|x_{j}-\bar{x}\|}{\sin^{-1}(\beta_{k}-\alpha_{k})}\\
\Rightarrow\frac{\|x_{k}-\bar{x}\|}{\|x_{j}-\bar{x}\|} & \leq & \frac{\kappa\epsilon}{\sin^{-1}(\bar{\beta}-\bar{\alpha})}.
\end{eqnarray*}
In the case where $K_{l}$ have the SOSH property at $\bar{x}$, we
can replace the $\epsilon\|x_{j}-\bar{x}\|$ in \eqref{eq:ineq-B}
by $M\|x_{j}-\bar{x}\|^{2}$. Reworking through the inequalities gives
\[
\frac{\|x_{k}-\bar{x}\|}{\|x_{j}-\bar{x}\|^{2}}\leq\frac{\kappa M}{\sin^{-1}(\bar{\beta}-\bar{\alpha})}\mbox{ if }k>j>i>I.
\]
In view of the fact that $\|x_{j}-\bar{x}\|\leq\|x_{i}-\bar{x}\|$
from condition (2) of Step 2 of Algorithm \ref{alg:Mass-proj-alg},
we have the result we need. 
\end{proof}
We list a few corollaries that are straightforward from Theorem \ref{thm:conv-SIP}.
The $\bar{p}$-term superlinear convergence in Corollary \ref{cor:p-term-suplin}
was the original conclusion in \cite{cut_Pang12}. Corollary \ref{cor:fast-sip}
shows that the convergence can be much faster for smooth problems.
\begin{cor}
\label{cor:p-term-suplin}($\bar{p}$-term superlinear convergence
for SIP) Suppose Algorithm \ref{alg:Mass-proj-alg} with parameter
$\bar{p}$ produces a sequence $\{x_{i}\}$ that converges to a point
$\bar{x}\in K$, and the convergence is not finite, and condition
(1) of Theorem \ref{thm:conv-SIP} holds. 

From condition (1) and Lemma \ref{lem:loc-metr-ineq-condn}, $\{K_{l}\}_{l=1}^{r}$
satisfies the local metric inequality at $\bar{x}$ with some constant,
say $\bar{\kappa}$. Let $\bar{\alpha}$ be such that 
\[
\bar{\alpha}<\sin^{-1}(1/\bar{\kappa}),
\]
and suppose that the parameter $\bar{p}$ in Algorithm \ref{alg:Mass-proj-alg}
is sufficiently large so that for any $\bar{p}$ unit vectors in $\mathbb{R}^{n}$,
there are two vectors such that the angle between them is at most
$\bar{\alpha}$. Suppose also that $S_{i}=\{\max(1,i-\bar{p}),\dots,i\}\times\{1,\dots,r\}$
(which implies that $S_{i}$ has $(\bar{p}+1)r$ elements). Then $\{x_{i}\}$
converge $\bar{p}$-term superlinearly to $\bar{x}$, i.e., 
\begin{equation}
\lim_{i\to\infty}\frac{\|x_{i+\bar{p}}-\bar{x}\|}{\|x_{i}-\bar{x}\|}=0.\label{eq:p-term}
\end{equation}
If in addition the sets $K_{l}$ have the SOSH property at $\bar{x}$
for all $l\in\{1,\dots,r\}$, then $\{x_{i}\}$ converge $\bar{p}$-term
quadratically to $\bar{x}$, i.e., 
\begin{equation}
\limsup_{i\to\infty}\frac{\|x_{i+\bar{p}}-\bar{x}\|}{\|x_{i}-\bar{x}\|^{2}}<\infty.\label{eq:p-term-1}
\end{equation}

\begin{cor}
\label{cor:fast-sip}(Fast convergence for smooth SIP) Suppose Algorithm
\ref{alg:Mass-proj-alg} with parameter $\bar{p}$ is such that $S_{i}=\{i\}\times\{1,\dots,r\}$
(which implies that $S_{i}$ has $r$ elements) produces a sequence
$\{x_{i}\}$ that converges to a point $\bar{x}\in K$, and the convergence
is not finite, and condition (1) of Theorem \ref{thm:conv-SIP} holds.
If $N_{K_{l}}(\bar{x})$ contains only one nonzero direction for all
$l\in\{1,\dots,r\}$, then the convergence of $\{x_{i}\}$ to $\bar{x}$
is superlinear, i.e., $\lim_{i\to\infty}\frac{\|x_{i+1}-\bar{x}\|}{\|x_{i}-\bar{x}\|}=0$.
If in addition the sets $K_{l}$ have the SOSH property at $\bar{x}$,
then the convergence of $\{x_{i}\}$ to $\bar{x}$ is quadratic, i.e.,
$\limsup_{i\to\infty}\frac{\|x_{i+1}-\bar{x}\|}{\|x_{i}-\bar{x}\|^{2}}<\infty$.
\end{cor}
\end{cor}
\begin{proof}
Due to the fact that the graph of the normal cone mapping $N_{C}:\mathbb{R}^{n}\rightrightarrows\mathbb{R}^{n}$
is closed for any closed convex set $C\subset\mathbb{R}^{n}$, the
unit vectors of the normals obtained by projecting $x_{i}$ onto each
$K_{l}$ converge to the only direction of unit length in $N_{K_{l}}(\bar{x})$
for $l\in\{1,\dots,r\}$. 

We now examine Statement (a) of Theorem \ref{thm:conv-SIP}. Choose
any $\bar{\alpha}>0$. As a consequence of the outer semicontinuity
of the mapping $N_{K_{l}}:\mathbb{R}^{n}\rightrightarrows\mathbb{R}^{n}$
at $\bar{x}$ for all $l\in\{1,\dots,r\}$, there is an $I_{1}$ such
that if $i>I_{1}$, then $\angle a_{i}^{(l_{i+1}^{*})}0a_{i+1}^{(l_{i+1}^{*})}<\bar{\alpha}$.
Hence for any $\epsilon>0$, we can increase $I_{1}$ if necessary
so that if $i>I_{1}$, then $\frac{\|x_{i+1}-\bar{x}\|}{\|x_{i}-\bar{x}\|}\leq\epsilon$.
A similar conclusion holds for quadratic convergence.
\end{proof}
We make another remark about higher order $\bar{p}$-term convergence.
\begin{rem}
(Higher order $\bar{p}$-term convergence) We note that $\bar{p}$-term
quadratic convergence \eqref{eq:p-term-1} implies 
\begin{equation}
\limsup_{i\to\infty}\frac{\|x_{i+2\bar{p}}-\bar{x}\|}{\|x_{i}-\bar{x}\|^{4}}<\infty,\label{eq:quartic-conv}
\end{equation}
which would be $(2\bar{p})$-term quartic convergence. So if we do
not have bounds for the parameter $\bar{p}$, then the degree of the
denominator of the term in \eqref{eq:quartic-conv} can be set arbitrarily
high, and the limit superior can also taken to be zero. It is therefore
important to bound $\bar{p}$. 
\begin{rem}
(Algorithmic consequences of Theorem \ref{thm:conv-SIP}) An insight
obtained from Theorem \ref{thm:conv-SIP} for how an algorithm might
run in practice is that after a few iterations, the unit normal vectors
of the halfspaces generated by the projection process can be compared.
If a unit normal vector of an old halfspace is close enough to a newer
one, then the old halfspace can be removed for the next QP subproblem
of projecting onto a polyhedron.
\end{rem}
\end{rem}

\section{\label{sec:CIP}A subgradient algorithm for the CIP}

In this section, we show how the ideas for the SIP can be transferred
to the CIP. We write down the SGQP (subgradient quadratic programming)
algorithm for solving the CIP \eqref{eq:CIP}.
\begin{algorithm}
\label{thm:SHQP-alg}(SGQP algorithm for the CIP) For a convex function
$f:\mathbb{R}^{n}\to\mathbb{R}$ and starting iterate $x_{0}$, we
find a point $\bar{x}$ such that $f(\bar{x})\leq0$.

\textbf{\uline{Step 0}}: Set $i=0$ and let  $\bar{p}$ be a fixed
positive integer.

\textbf{\uline{Step 1}}: Find $y_{i}\in\partial f(x_{i})$ and
define the halfspace $H_{i}^{\leq}\subset\mathbb{R}^{n}$ by 
\begin{equation}
H_{i}^{\leq}:=\{x\in\mathbb{R}^{n}\mid f(x_{i})+\left\langle y_{i},x-x_{i}\right\rangle \leq0\}.\label{eq:H_i^leq}
\end{equation}

\textbf{\uline{Step 2}}: Find $x_{i+1}$ such that 

$\qquad$(1) $x_{i+1}\in F_{i}$, where $S_{i}$ is a subset of $\{\max(i-\bar{p},0),\dots,i\}$
and ${F_{i}:=\cap_{k\in S_{i}}H_{k}^{\leq}}$.

$\qquad$(2) $\|x_{i+1}-c\|\leq\|x_{i}-c\|$ for all $c\in f^{-1}((-\infty,0])$.

$\qquad$(3) $x_{i}-x_{i+1}$ lies in $\conv\{y_{k}\}_{k\in S_{i}}$,
i.e., $x_{i}-x_{i+1}$ lies in the cone generated by the convex hull
of the normals of the halfspaces $\{H_{k}^{\leq}\}_{k\in S_{i}}$. 

\textbf{\uline{Step 3}}: Set $i\leftarrow i+1$, and go back to
step 1.
\end{algorithm}
When $S_{i}\equiv\{i\}$ in Algorithm \ref{thm:SHQP-alg}, the iterate
$x_{i+1}$ can be calculated to be $x_{i+1}=x_{i}-\frac{f(x_{i})}{\|y_{i}\|^{2}}y_{i}$.
 It is easy to check that Algorithm \ref{thm:SHQP-alg} converges
linearly when $S_{i}\equiv\{i\}$ for $f:\mathbb{R}^{2}\to\mathbb{R}$
defined by $f(x_{1},x_{2})=\max(2x_{1}-x_{2},2x_{2}-x_{1})$ exhibits
the zigzagging behavior typical of alternating projections, but converges
in finitely many iterations when $S_{i}\supset\{i,i-1\}$ for large
$i$.

Remark \ref{rem:step-2-relax} also applies to Step 2 of Algorithm
\ref{thm:SHQP-alg}; We can take $x_{i+1}$ to be the projection of
$x_{i}$ onto $F_{i}$ in Step 2 of \ref{thm:SHQP-alg}.
\begin{thm}
(Basic convergence for the CIP) Suppose $f:\mathbb{R}^{n}\to\mathbb{R}$
is convex, and $f^{-1}((-\infty,0])\neq\emptyset$. Then Algorithm
\ref{thm:SHQP-alg} for any parameter $\bar{p}$ and $S_{i}\supset\{i\}$
converges to a point $\bar{x}$ such that $f(\bar{x})\leq0$.\end{thm}
\begin{proof}
By Theorem \ref{thm:Fejer-ppty}(1), the sequence of iterates $\{x_{i}\}$
is bounded, and has a convergent subsequence. Suppose it has a cluster
point $\bar{x}$. Seeking a contradiction, suppose $f(\bar{x})>0$.
Since $0\notin\partial f(\bar{x})$, let $\gamma:=\sup\{\|y\|:y\in\partial f(\bar{x})\}$.
If $x_{i}$ is sufficiently close to $\bar{x}$, then for the choice
$y_{i}\in\partial f(x_{i})$, the distance of $x_{i}$ to the halfspace
$H_{i}^{\leq}$ defined in \eqref{eq:H_i^leq} is $f(x_{i})/\|y_{i}\|$.
By the outer semicontinuity of the subdifferential and the continuity
of $f(\cdot)$, the value $f(x_{i})/\|y_{i}\|$ is in turn bounded
from below by $\frac{f(\bar{x})}{2\gamma}$ if $x_{i}$ is sufficiently
close to $\bar{x}$. This implies that 
\begin{equation}
\|x_{i}-x_{i+1}\|\geq\frac{f(\bar{x})}{2\gamma}.\label{eq:succ-iter-l-bdd}
\end{equation}

We simplify the statements in the proof by letting $C$ to be $f^{-1}((-\infty,0])$.
Consider iterates $x_{i}$ and $x_{i+1}$. Since $x_{i+1}$ is the
projection of $x_{i}$ onto a set containing $C$, we have $\left\langle x_{i}-x_{i+1},P_{C}(x_{i})-x_{i+1}\right\rangle \leq0$.
This inequality implies that 
\[
\|x_{i}-x_{i+1}\|^{2}+\|x_{i+1}-P_{C}(x_{i})\|^{2}\leq\|x_{i}-P_{C}(x_{i})\|^{2},
\]
which in turn gives 
\begin{eqnarray*}
d(x_{i+1},C)^{2} & \leq & \|x_{i+1}-P_{C}(x_{i})\|^{2}\\
 & \leq & \|x_{i}-P_{C}(x_{i})\|^{2}-\|x_{i}-x_{i+1}\|^{2}\\
 & = & d(x_{i},C)^{2}-\|x_{i}-x_{i+1}\|^{2}.
\end{eqnarray*}
It follows from the continuity of $d(\cdot,C)$ and \eqref{eq:succ-iter-l-bdd}
that if $x_{i}$ were sufficiently close to $\bar{x}$, then $d(x_{i+1},C)^{2}\leq d(\bar{x},C)^{2}-[\frac{f(\bar{x})}{2\gamma}]^{2}$.
This fact and Theorem \ref{thm:Fejer-ppty}(1) contradicts the assumption
that $\bar{x}$ is a cluster point of $\{x_{i}\}$. Therefore, the
cluster points of $\{x_{i}\}$ must belong to $C$. By Theorem \ref{thm:Fejer-ppty}(2),
we conclude that $\{x_{i}\}$ converges to some point $\bar{x}$ in
$C$. In other words, $\{x_{i}\}$ converges to some $\bar{x}$ such
that $f(\bar{x})\leq0$. 
\end{proof}
We now prove a few intermediate inequalities useful for Theorem \ref{thm:superlin-CIP}.
\begin{lem}
\label{lem:2-int-results}(Intermediate inequalities) Let $f:\mathbb{R}^{n}\to\mathbb{R}$
be a convex function. Choose $\bar{x}\in\mathbb{R}^{n}$ such that
$f(\bar{x})=0$ and $0\notin\partial f(\bar{x})$. Let $\gamma_{1}$
be such that $\gamma_{1}<d(0,\partial f(\bar{x}))$ and $\gamma_{2}<\frac{d(0,\partial f(\bar{x}))}{\sup_{y\in\partial f(\bar{x})}\|y\|}$.
Then there is some $\epsilon>0$ such that for all $x$ such that
$\|x-\bar{x}\|<\epsilon$, $f(x)>0$ and $d\big(\frac{x-\bar{x}}{\|x-\bar{x}\|},N_{f^{-1}((-\infty,0])}(\bar{x})\big)<\epsilon$,
we have 
\begin{enumerate}
\item $\frac{f(x)}{\|x-\bar{x}\|}\geq\gamma_{1}$. 
\item $\left\langle \frac{y}{\|y\|},\frac{x-\bar{x}}{\|x-\bar{x}\|}\right\rangle >\gamma_{2}$
for all $y\in\partial f(x)$. 
\end{enumerate}
\end{lem}
\begin{proof}
By the convexity of $f(\cdot)$, we have 
\begin{eqnarray}
f(x) & \geq & f(\bar{x})+\sup_{z\in\partial f(\bar{x})}\left\langle z,x-\bar{x}\right\rangle \nonumber \\
\frac{f(x)-f(\bar{x})}{\|x-\bar{x}\|} & \geq & \sup_{z\in\partial f(\bar{x})}\left\langle z,\frac{x-\bar{x}}{\|x-\bar{x}\|}\right\rangle .\label{eq:lem-ineq2}
\end{eqnarray}
Since $d\big(\frac{x-\bar{x}}{\|x-\bar{x}\|},N_{f^{-1}((-\infty,0])}(\bar{x})\cap\partial\mathbb{B}\big)<\epsilon$,
we can find $v$ such that $\|v\|=1$, $\|v-\frac{x-\bar{x}}{\|x-\bar{x}\|}\|<\epsilon$
and $v\in N_{f^{-1}((-\infty,0])}(\bar{x})$. Let $w=v-\frac{x-\bar{x}}{\|x-\bar{x}\|}$.
We now make use of the well known fact that $\lip f(\bar{x})=\max_{z\in\partial f(\bar{x})}\|z\|$.
For any $z\in\partial f(\bar{x})$, we have 
\[
-\left\langle z,w\right\rangle \geq-\|w\|\lip f(\bar{x})\geq-\epsilon\,\lip f(\bar{x}).
\]
Since $N_{f^{-1}((-\infty,0])}(\bar{x})=\pos(\partial f(\bar{x}))$
by \cite[Proposition 10.3]{RW98}, there is some $z^{\prime}\in\partial f(\bar{x})$
such that $v=\frac{z^{\prime}}{\|z^{\prime}\|}$. We have 
\begin{eqnarray}
\sup_{z\in\partial f(\bar{x})}\left\langle z,\frac{x-\bar{x}}{\|x-\bar{x}\|}\right\rangle  & = & \sup_{z\in\partial f(\bar{x})}[\left\langle z,v\right\rangle -\left\langle z,w\right\rangle ]\nonumber \\
 & \geq & \sup_{z\in\partial f(\bar{x})}[\left\langle z,v\right\rangle -\epsilon\,\lip f(\bar{x})]\nonumber \\
 & \geq & \frac{\left\langle z^{\prime},z^{\prime}\right\rangle }{\|z^{\prime}\|}-\epsilon\,\lip f(\bar{x})\nonumber \\
 & = & \|z^{\prime}\|-\epsilon\,\lip f(\bar{x})\nonumber \\
 & \geq & d\big(0,\partial f(\bar{x})\big)-\epsilon\,\lip f(\bar{x}).\label{eq:lem-ineq3}
\end{eqnarray}
Since $0\notin\partial f(\bar{x})$, $d(0,\partial f(\bar{x}))>0$.
We can reduce $\epsilon>0$ if necessary, so by combining \eqref{eq:lem-ineq2}
and \eqref{eq:lem-ineq3}, we have conclusion (1). 

Let $y\in\partial f(x)$. We have 
\begin{eqnarray}
f(\bar{x}) & \geq & f(x)+\left\langle y,\bar{x}-x\right\rangle \nonumber \\
\Rightarrow\left\langle \frac{y}{\|y\|},\frac{x-\bar{x}}{\|x-\bar{x}\|}\right\rangle  & \geq & \frac{f(x)}{\|y\|\|x-\bar{x}\|}.\label{eq:lem-ineq1}
\end{eqnarray}

The outer semicontinuity of $\partial f(\cdot)$ implies that for
any $\delta>0$, we can reduce $\epsilon$ if necessary so that $\|y\|\leq\sup_{y^{\prime}\in\partial f(\bar{x})}\|y^{\prime}\|+\delta$
whenever $y\in\partial f(x)$ and $\|x-\bar{x}\|\leq\epsilon$. By
combining this observation to \eqref{eq:lem-ineq1} together with
conclusion (1), we get conclusion (2).
\end{proof}
We present our result on the convergence of the CIP. 
\begin{thm}
\label{thm:superlin-CIP}(Convergence rates for the CIP) Suppose that
Algorithm \ref{thm:SHQP-alg} with parameter $\bar{p}$ for a convex
function $f:\mathbb{R}^{n}\to\mathbb{R}$ produces a sequence $\{x_{i}\}$
that converges to a point $\bar{x}\in f^{-1}(0)$ such that $0\notin\partial f(\bar{x})$,
and the convergence to $\bar{x}$ is not finite. Let $\bar{\alpha}>0$
be such that 
\[
\bar{\alpha}<\sin^{-1}\left(\frac{d(0,\partial f(\bar{x}))}{\sup_{y\in\partial f(\bar{x})}\|y\|}\right).
\]
Suppose that the $\bar{p}$ in Algorithm \ref{thm:SHQP-alg} is sufficiently
large so that for any $\bar{p}$ unit vectors in $\mathbb{R}^{n}$,
there are two vectors such that the angle between them is at most
$\bar{\alpha}$. 

(a) For any $\epsilon>0$, there is an $I>0$ such that if $I<i<k$
and $\angle y_{j}0y_{k}\leq\bar{\alpha}$ for some $j\in\{i,i-1,\dots,k-1\}\cap S_{k}$,
then 
\[
\frac{\|x_{k}-\bar{x}\|}{\|x_{i}-\bar{x}\|}\leq\frac{\|x_{k}-\bar{x}\|}{\|x_{j}-\bar{x}\|}\leq\epsilon.
\]

(b) If in addition $f$ is strongly semismooth at $\bar{x}$, then
there are $M>0$ and $I>0$ such that if $I<i<k$ and $\angle y_{j}0y_{k}\leq\bar{\alpha}$
for some $j\in\{i,i-1,\dots,k-1\}\cap S_{k}$, then 
\begin{equation}
\frac{\|x_{k}-\bar{x}\|}{\|x_{i}-\bar{x}\|^{2}}\leq\frac{\|x_{k}-\bar{x}\|}{\|x_{j}-\bar{x}\|^{2}}<M.\label{eq:p-term-quad}
\end{equation}
\end{thm}
\begin{proof}
The proof has quite a few parts.

\textbf{\uline{Part 1}}\textbf{: Any cluster point of $\{\frac{x_{i}-\bar{x}}{\|x_{i}-\bar{x}\|}\}$
is in $N_{f^{-1}((-\infty,0])}(\bar{x})=\pos(\partial f(\bar{x}))$.}

This part of the proof contains ideas in the proof of \cite[Proposition 5.8]{cut_Pang12},
but is much simpler. We use the variables $j$ and $k$ to be running
variables incompatible with (a) and (b) of the main result. By the
design of step 2, part (3) of Algorithm \ref{thm:SHQP-alg}, we have
\[
x_{i+1}=x_{i}-\sum_{k=\max(1,i-\bar{p})}^{i}\lambda_{i,k}y_{k},
\]
where $\lambda_{i,k}\geq0$ for all $i$ and $k$ such that $\max(1,i-\bar{p})\leq k\leq i$,
and $\lambda_{i,k}=0$ if $k\notin S_{i}$. For $j>i$, we then have
\[
x_{i}-x_{j}=\sum_{s=i}^{j-1}\sum_{k=\max(1,s-\bar{p})}^{s}\lambda_{s,k}y_{k}.
\]
We have
\begin{equation}
x_{i}-\bar{x}=\lim_{j\to\infty}\sum_{s=i}^{j-1}\sum_{k=\max(1,s-\bar{p})}^{s}\lambda_{s,k}y_{k}.\label{eq:x-i-minus-x-bar}
\end{equation}
We now show that the convergence of \eqref{eq:x-i-minus-x-bar} is
absolute so that an infinite sum notation is justified. By the outer
semicontinuity of the subdifferential mapping $\partial f(\cdot)$,
for any $\epsilon>0$, we can find $I$ large enough so that $\partial f(x_{i})\subset\partial f(\bar{x})+\epsilon\mathbb{B}$
for all $i>I-\bar{p}$. Since $0\notin\partial f(\bar{x})$, we can
choose $\epsilon$ small enough so that $0\notin\partial f(\bar{x})+\epsilon\mathbb{B}$.
The set $\pos(\partial f(\bar{x})+\epsilon\mathbb{B})$ is a pointed
cone, so by Lemma \ref{lem:pointed-cone-modulus}, there is a constant
$m>0$ such that 
\[
m\sum_{s=i}^{j-1}\sum_{k=\max(1,s-\bar{p})}^{s}\lambda_{s,k}\|y_{k}\|\leq\left\Vert \sum_{s=i}^{j-1}\sum_{k=\max(1,s-\bar{p})}^{s}\lambda_{s,k}y_{k}\right\Vert =\|x_{i}-x_{j}\|.
\]
Taking the limits as $j\to\infty$, we have 
\[
\sum_{s=i}^{\infty}\sum_{k=\max(1,s-\bar{p})}^{s}\lambda_{s,k}\|y_{k}\|\leq\frac{1}{m}\|x_{i}-\bar{x}\|<\infty,
\]
which shows that the convergence of \eqref{eq:x-i-minus-x-bar} is
absolute.

For any $\epsilon>0$, we can always choose $i$ large enough so that
$y_{k}\in\partial f(\bar{x})+\epsilon\mathbb{B}$ for all $k\geq i-\bar{p}$.
This means that 
\[
x_{i}-\bar{x}=\sum_{s=i}^{\infty}\sum_{k=\max(1,s-\bar{p})}^{s}\lambda_{s,k}y_{k}\in\pos\big(\partial f(\bar{x})+\epsilon\mathbb{B}\big).
\]
Since $\pos(\partial f(\bar{x}))$ is a pointed cone, the set $\pos(\partial f(\bar{x})+\epsilon\mathbb{B})\cap\partial\mathbb{B}$
converges to $\pos(\partial f(\bar{x}))\cap\partial\mathbb{B}$ in
the Pompieu Hausdorff distance as $\epsilon\searrow0$. So the cluster
points of $\{\frac{x_{i}-\bar{x}}{\|x_{i}-\bar{x}\|}\}$ lie in $\pos(\partial f(\bar{x}))$.
Since\textbf{ }$\pos(\partial f(\bar{x}))$\textbf{ }equals \textbf{$N_{f^{-1}((-\infty,0])}(\bar{x})$}
by \cite[Proposition 10.3]{RW98}, we are done.

\textbf{\uline{Part 2}}\textbf{: Applying Lemma \ref{lem:2-int-results}.}

By applying Lemma \ref{lem:2-int-results} and part 1, we deduce that
if $\gamma_{1}<d(0,\partial f(\bar{x}))$ and $\gamma_{2}<\frac{d(0,\partial f(\bar{x}))}{\sup_{y\in\partial f(\bar{x})}\|y\|}$,
then there is an $I>0$ such that\begin{subequations} 
\begin{eqnarray}
 &  & \frac{f(x_{i})}{\|x_{i}-\bar{x}\|}>\gamma_{1}\label{eq:part2-1}\\
 & \mbox{ and } & \left\langle \frac{x_{i}-\bar{x}}{\|x_{i}-\bar{x}\|},\frac{y_{i}}{\|y_{i}\|}\right\rangle >\gamma_{2}\mbox{ for all }i>I.\label{eq:part2-2}
\end{eqnarray}
\end{subequations}

\textbf{\uline{Part 3}}\textbf{: Defining and evaluating }$\limsup_{j\to\infty}\lambda_{j}$
\textbf{and $\limsup_{j\to\infty}\frac{\lambda_{j}}{\|\bar{x}-x_{j}\|}$.}

 Define $f_{j}:\mathbb{R}^{n}\to\mathbb{R}$ to be $f_{j}(x):=f(x_{j})+\left\langle y_{j},x-x_{j}\right\rangle $.
In view of \eqref{eq:part2-1}, we have $f_{j}(x_{j})=f(x_{j})\geq\gamma_{1}\|\bar{x}-x_{j}\|$
for $j$ large enough. Let $\lambda_{j}\in\mathbb{R}$ be such that
$f_{j}(\lambda_{j}(x_{j}-\bar{x})+\bar{x})=0$. By the convexity of
$f(\cdot)$, it is clear that $\lambda_{j}\geq0$. Since $f(x_{j})>0$,
we have $\lambda_{j}<1$. By the semismoothness of $f$ at $\bar{x}$
and applying \eqref{eq:part2-1}, we have 
\begin{eqnarray}
(1-\lambda_{j})f_{j}(\bar{x})+\lambda_{j}f_{j}(x_{j}) & = & f_{j}\big(\lambda_{j}(x_{j}-\bar{x})+\bar{x}\big)=0\nonumber \\
\Rightarrow\lambda_{j} & = & \frac{-f_{j}(\bar{x})}{f_{j}(x_{j})-f_{j}(\bar{x})}\label{eq:lambda-to-zero}\\
 & = & \frac{-[f(x_{j})+\left\langle y_{j},\bar{x}-x_{j}\right\rangle ]}{f(x_{j})-[f(x_{j})+\left\langle y_{j},\bar{x}-x_{j}\right\rangle ]}\nonumber \\
 & = & \frac{f(\bar{x})-f(x_{j})-\left\langle y_{j},\bar{x}-x_{j}\right\rangle }{f(x_{j})+f(\bar{x})-f(x_{j})-\left\langle y_{j},\bar{x}-x_{j}\right\rangle }\nonumber \\
 & \leq & \frac{o(\|\bar{x}-x_{j}\|)}{\gamma_{1}\|\bar{x}-x_{j}\|+o(\|\bar{x}-x_{j}\|)}\nonumber \\
 & = & o(1).\nonumber 
\end{eqnarray}

In other words, $\limsup_{j\to\infty}\lambda_{j}=0$. In the case
where $f$ is strongly semismooth at $\bar{x}$, we can repeat the
calculations to get 
\[
\lambda_{j}\leq\frac{O(\|\bar{x}-x_{j}\|^{2})}{\gamma_{1}\|\bar{x}-x_{j}\|+O(\|\bar{x}-x_{j}\|^{2})}=O(\|\bar{x}-x_{j}\|),
\]
or $\limsup_{j\to\infty}\frac{\lambda_{j}}{\|\bar{x}-x_{j}\|}<\infty$.

\textbf{\uline{Part 4}}\textbf{: Bounding $\frac{\|x_{k}-\bar{x}\|}{\|x_{k}-x_{k}^{+}\|}$.}

We shall let $x_{k}^{+}$ be the point $x_{k}-\frac{f(x_{k})}{\|y_{k}\|^{2}}y_{k}$.
The point $x_{k}^{+}$ is also the projection of $x_{k}$ onto $H_{k}^{\leq}$.
The distance $\|x_{k}-x_{k}^{+}\|$ is easily calculated to be $\frac{|f(x_{k})|}{\|y_{k}\|}$.
The distance from $x_{k}$ to the hyperplane with normal $y_{k}$
passing through $\bar{x}$ can be calculated to be $\frac{1}{1-\lambda_{k}}\frac{|f(x_{k})|}{\|y_{k}\|}$. 

We now show that for any $\gamma_{2}<\frac{d(0,\partial f(\bar{x}))}{\sup_{y\in\partial f(\bar{x})}\|y\|}$,
we can find $I$ such that if $k>I$, then the angle 
\begin{equation}
\theta_{k}:=\angle x_{k}^{+}x_{k}\bar{x}\label{eq:def-theta}
\end{equation}
 is such that $\theta_{k}\leq\cos^{-1}\gamma_{2}$ by \eqref{eq:part2-2}.
Since $x_{k}-x_{k}^{+}$ is a positive multiple of $y_{k}$, by \eqref{eq:part2-1},
\[
\cos\theta_{k}=\left\langle \frac{x_{k}-\bar{x}}{\|x_{k}-\bar{x}\|},\frac{y_{i}}{\|y_{i}\|}\right\rangle >\gamma_{2},
\]
which gives us what we need. 

Now
\begin{eqnarray*}
(1-\lambda_{k})\gamma_{2}\|x_{k}-\bar{x}\| & \leq & (1-\lambda_{k})\|x_{k}-\bar{x}\|\cos\theta_{k}\\
 & = & [1-\lambda_{k}]\frac{1}{1-\lambda_{k}}\frac{|f(x_{k})|}{\|y_{k}\|}\\
 & = & \|x_{k}-x_{k}^{+}\|.
\end{eqnarray*}
Therefore, if $k$ is large enough, we have 
\begin{equation}
\|x_{k}-\bar{x}\|\leq K\|x_{k}-x_{k}^{+}\|,\label{eq:CIP-lin-reg}
\end{equation}
where $K=1.1/\gamma_{2}$. 

\textbf{\uline{Part 5}}\textbf{: Evaluating $\frac{\|x_{k}-\bar{x}\|}{\|x_{j}-\bar{x}\|}$
and $\frac{\|x_{k}-\bar{x}\|}{\|x_{j}-\bar{x}\|^{2}}$, and wrapping
up.}

We consider the two-dimensional space containing $\bar{x}$, $\bar{x}+y_{j}$
and $\bar{x}+y_{k}$. Let $x_{k}^{++}$ be the point on the line containing
$\Pi x_{k}$ and $\Pi x_{k}^{+}$ such that $\angle(\Pi x_{k})\bar{x}(\Pi x_{k}^{++})=\pi/2$.
We shall call the projections of $x_{k}$, $x_{k}^{+}$ and $x_{k}^{++}$
onto this 2d space to be $\Pi x_{k}$, $\Pi x_{k}^{+}$ and $\Pi x_{k}^{++}$.
We note that the distance from $\bar{x}$ to $A_{j}:=\partial H_{j}^{\leq}$,
the boundary of $H_{j}^{\leq}$, is bounded above by $\lambda_{j}\|x_{j}-\bar{x}\|$
because $\lambda_{j}(x_{j}-\bar{x})+\bar{x}\in A_{j}$ and $\|[\lambda_{j}(x_{j}-\bar{x})+\bar{x}]-\bar{x}\|=\lambda_{j}\|x_{j}-\bar{x}\|$.
We can refer back to Figure \ref{fig:geom1}, but replace $a_{k}^{(l_{k}^{*})}$
by $y_{k}$ and $a_{j}^{(l)}$ by $y_{j}$.

Making use of (1) in step 2 of Algorithm \ref{thm:SHQP-alg}, we can
use elementary geometry (in the same manner as in the proof of \eqref{eq:ineq-C})
to prove the inequality 
\begin{equation}
\|\Pi x_{k}-\bar{x}\|\leq\frac{d(\bar{x},A_{j})}{\sin(\beta_{k}-\alpha_{k})}\leq\frac{\lambda_{j}\|x_{j}-\bar{x}\|}{\sin(\beta_{k}-\alpha_{k})},\label{eq:elem-geom}
\end{equation}
where $\beta_{k}:=\angle(\Pi x_{k})\bar{x}(\Pi x_{k}^{++})$ and $\alpha_{k}=\angle y_{j}0y_{k}$.
The angle $\beta_{k}$ is bounded from below by 
\[
\beta_{k}=\angle(\Pi x_{k})\bar{x}(\Pi x_{k}^{++})\geq\angle x_{k}\bar{x}x_{k}^{++}=\frac{\pi}{2}-\angle x_{k}^{+}x_{k}\bar{x}=\frac{\pi}{2}-\theta_{k},
\]
where $\theta_{k}$ was defined in \eqref{eq:def-theta}, and is bounded
from above by $\cos^{-1}\gamma_{2}$ for large $k$. For any choice
of $\bar{\alpha}$, we can make $\gamma_{2}$ close enough to $\frac{d(0,\partial f(\bar{x}))}{\sup_{y\in\partial f(\bar{x})}\|y\|}$
so that for all $i$ large enough, we have
\[
\beta_{k}\geq\bar{\beta}>\bar{\alpha},
\]
where $\bar{\beta}:=\sin^{-1}\gamma_{2}$. Combining with the assumption
that $\alpha_{k}\leq\bar{\alpha}$, \eqref{eq:elem-geom} gives 
\begin{equation}
\|\Pi x_{k}-\bar{x}\|\leq\frac{\lambda_{j}\|x_{j}-\bar{x}\|}{\sin(\bar{\beta}-\bar{\alpha})}\mbox{ for all }i\mbox{ large enough.}\label{eq:elem-geom2}
\end{equation}
We have, by \eqref{eq:CIP-lin-reg} and \eqref{eq:elem-geom2},
\begin{eqnarray*}
\|x_{k}-\bar{x}\| & \leq & K\|x_{k}-x_{k}^{+}\|\\
 & = & K\|\Pi x_{k}-\Pi x_{k}^{+}\|\\
 & \leq & K\|\Pi x_{k}-\bar{x}\|\\
 & \leq & K\frac{\lambda_{j}\|x_{j}-\bar{x}\|}{\sin(\bar{\beta}-\bar{\alpha})}\\
\Rightarrow\frac{\|x_{k}-\bar{x}\|}{\|x_{j}-\bar{x}\|} & \leq & \frac{K\lambda_{j}}{\sin(\bar{\beta}-\bar{\alpha})}.
\end{eqnarray*}
Making use of the fact that $\lambda_{i}\searrow0$ in \eqref{eq:lambda-to-zero},
for any $\epsilon>0$, we can choose $I$ large enough so that $\frac{K\lambda_{i}}{\sin(\bar{\beta}-\bar{\alpha})}<\epsilon$
for all $i>I$. Thus the conclusion that we seek holds. In the case
where $f(\cdot)$ is strongly semismooth, we can make use of the fact
that $\limsup_{i\to\infty}\frac{\lambda_{i}}{\|x_{i}-\bar{x}\|}$
is finite in Part 3, say with value $M_{1}$, to get, for all $j$,
$k$ large enough, that 
\[
\frac{\|x_{k}-\bar{x}\|}{\|x_{j}-\bar{x}\|^{2}}\leq\frac{K(M_{1}+1)}{\sin(\bar{\beta}-\bar{\alpha})}.
\]
In view of (2) of Step 2 in Algorithm \ref{thm:SHQP-alg}, we have
$\|x_{j}-\bar{x}\|\leq\|x_{i}-\bar{x}\|$. Thus our claim is proved.
\end{proof}
\begin{rem}
(Applicability of Theorem \ref{thm:superlin-CIP} to SIP) The SIP
\eqref{eq:SIP} is a special case of the CIP \eqref{eq:CIP}, since
the problem of finding a point in the intersection of $r$ closed
convex sets $K_{1},\dots,K_{r}$ can be regarded as the problem of
finding a point $x$ such that $f(x):=\max_{l=1}^{r}d(x,K_{l})\leq0$.
So at first glance, Theorem \ref{thm:superlin-CIP} seems to be applicable
for the SIP, which would be stronger than the main result in \cite{cut_Pang12}.
However, the condition $0\notin\partial f(\bar{x})$ in Theorem \ref{thm:superlin-CIP}
is violated for the SIP.
\end{rem}

When the parameter $\bar{p}$ is set to $0$ in Algorithm \ref{thm:SHQP-alg},
we get linear convergence as shown below.
\begin{thm}
\label{thm:linear-CIP}(Linear convergence for the CIP) Suppose that
Algorithm \ref{thm:SHQP-alg} for a convex function $f:\mathbb{R}^{n}\to\mathbb{R}$
with $S_{i}\equiv\{i\}$ for all $i$ produces a sequence $\{x_{i}\}$
that converges to a point $\bar{x}\in f^{-1}(0)$ such that $0\notin\partial f(\bar{x})$.
Then the convergence is at least linear.\end{thm}
\begin{proof}
Note that this result is already a consequence of \cite{Robinson_76_CIP},
but we include its proof here since it is quite easy. 

The first 2 parts of the proof of Theorem \ref{thm:superlin-CIP}
still apply. Next, we have $x_{i+1}=x_{i}-\frac{f(x_{i})}{\|y_{i}\|^{2}}y_{i}$.
Using \eqref{eq:part2-1}, we get 
\begin{equation}
\left\Vert \frac{f(x_{i})}{\|y_{i}\|^{2}}y_{i}\right\Vert =\frac{|f(x_{i})|}{\|y_{i}\|}\geq\frac{\gamma_{1}}{\max_{y\in\partial f(x_{i})}\|y\|}\|x_{i}-\bar{x}\|.\label{eq:lower-bdd-on-dist}
\end{equation}
By the outer semicontinuity of $\partial f(\cdot)$, there is some
$I$ such that 
\begin{equation}
\frac{\gamma_{1}}{\max_{y\in\partial f(x_{i})}\|y\|}>\frac{\gamma_{1}}{2\max_{y\in\partial f(\bar{x})}\|y\|},\label{eq:lower-bdd-max-den}
\end{equation}
and the constant on the right is positive. Also, by the property of
projections, the angle $\angle x_{i}x_{i+1}\bar{x}$ is obtuse. Together
with \eqref{eq:lower-bdd-on-dist} and \eqref{eq:lower-bdd-max-den},
we have 
\begin{eqnarray*}
\|x_{i+1}-\bar{x}\|^{2} & \leq & \|x_{i}-\bar{x}\|^{2}-\|x_{i}-x_{i+1}\|^{2}\\
 & = & \|x_{i}-\bar{x}\|^{2}-\bigg[\frac{|f(x_{i})|}{\|y_{i}\|}\bigg]^{2}\\
 & \leq & \|x_{i}-\bar{x}\|^{2}-\bigg[\frac{\gamma_{1}}{2\max_{y\in\partial f(\bar{x})}\|y\|}\bigg]^{2}\|x_{i}-\bar{x}\|^{2}\\
\frac{\|x_{i+1}-\bar{x}\|}{\|x_{i}-\bar{x}\|} & \leq & \sqrt{1-\bigg[\frac{\gamma_{1}}{2\max_{y\in\partial f(\bar{x})}\|y\|}\bigg]^{2}}<1.
\end{eqnarray*}
This shows that $\limsup_{i\to\infty}\frac{\|x_{i+1}-\bar{x}\|}{\|x_{i}-\bar{x}\|}<1$,
which is the linear convergence we seek.
\end{proof}
Theorems \ref{thm:superlin-CIP} and \ref{thm:linear-CIP} suggest
that as the parameter $\bar{p}$ increases, the constant of linear
convergence gets lower, from linear convergence for the case of $\bar{p}=0$
in Theorem \ref{thm:linear-CIP} to the case of superlinear convergence
in Theorem \ref{thm:superlin-CIP}. 

We show that the condition $0\notin\partial f(\bar{x})$ cannot be
dropped in Theorems \ref{thm:superlin-CIP} and \ref{thm:linear-CIP}.
\begin{example}
(Condition $0\notin\partial f(\bar{x})$ essential in Theorems \ref{thm:superlin-CIP}
and \ref{thm:linear-CIP}) Let $f:\mathbb{R}\to\mathbb{R}$ be a convex
function such that 
\[
f(x)=\begin{cases}
e^{-1/x} & \mbox{ if }x\in(0,0.5]\\
0 & \mbox{ if }x=0\\
f(-x) & \mbox{ if }x\in[-0.5,0).
\end{cases}
\]
The function $f$ is convex on $[0,0.5]$ because $f^{\prime\prime}(x)=e^{-1/x}\frac{1-2x}{x^{4}}$
if $x\in(0,0.5)$, and is positive in that region. It is clear that
$f^{-1}((-\infty,0])=\{0\}$. Consider Algorithm \ref{thm:SHQP-alg}
for any parameter of $\bar{p}$. If any iterate $x_{i}$ is in $(0,0.5)$,
then the next iterate $x_{i+1}$ is calculated to be 
\[
x_{i+1}=x_{i}-\frac{f(x_{i})}{f^{\prime}(x_{i})}=x_{i}-\frac{e^{-1/x_{i}}}{e^{-1/x_{i}}(1/x_{i}^{2})}=x_{i}-x_{i}^{2}.
\]
It is clear that $\lim_{i\to\infty}x_{i}=0$, and that $\lim_{i\to\infty}\frac{|x_{i+1}-0|}{|x_{i}-0|}=1$,
so there is no linear convergence.
\end{example}
Robinson \cite{Robinson_76_CIP} included his reasons for analyzing
the KCIP \eqref{eq:KCIP} for only the case when the domain of $f$
is $\mathbb{R}^{n}$ and not a Hilbert space in general, and these
arguments carry over to the main results here as well. 
\begin{rem}
(On CIP involving many smooth convex functions) The original CIP studied
in \cite{G-P98,G-P01} was, for convex functions $f_{l}:\mathbb{R}^{n}\to\mathbb{R}$
defined for $l\in\{1,\dots,r\}$, 
\begin{equation}
\mbox{Find }x\in\mathbb{R}^{n}\mbox{ s.t. }f_{l}(x)\leq0\mbox{ for all }l\in\{1,\dots,r\}.\label{eq:CIP-smooth}
\end{equation}
One can treat the problem above as a CIP in our setting by considering
\begin{equation}
f_{\max}(\cdot):=\max_{l\in\{1,\dots,r\}}f_{l}(\cdot).\label{eq:f-max}
\end{equation}
A natural question to ask is whether the analysis in this section
can be generalized if we had studied \eqref{eq:CIP-smooth} instead.
Unfortunately, Lemma \ref{lem:2-int-results} cannot be easily extended.
For example, consider $f_{l}:\mathbb{R}^{2}\to\mathbb{R}$ defined
by $f_{l}(x)=x_{l}$ (i.e., taking the $l$th coordinate). There is
no constant $\gamma_{1}>0$ such that for $l\in\{1,2\}$, if $f_{l}(x)>0$,
then $f_{l}(x)\geq\gamma_{1}\|x\|$. Such a constant has to exist
in order for for Part 3 of the proof of Theorem \ref{thm:superlin-CIP}
to go through.

If $f_{l}(\cdot)$ were smooth convex functions for $l\in\{1,\dots,r\}$,
we can analyze $f_{\max}$ in \eqref{eq:f-max} using the results
in this section to obtain $r$-term superlinear or $r$-term quadratic
convergence (defined in \eqref{eq:p-term} and \eqref{eq:p-term-1})
to a point $\bar{x}\in f_{\max}^{-1}((-\infty,0])$ if the subdifferential
of $f_{\max}$ at an iterate $x_{i}$ is taken to be $\nabla f_{l}(x_{i})$
for some $l\in\{1,\dots,r\}$. 
\end{rem}

\section{\label{sec:Infeasible}Infeasibility}

Closely related to the SIP is the \emph{Best Approximation Problem
}(BAP): For a Hilbert space $X$, a point $x_{0}\in X$ and $r$ closed
convex sets $K_{i}$ for $i=1,\dots,r$, find the closest point to
$x_{0}$ in $K:=\cap K_{i}$. That is, 
\begin{eqnarray}
\mbox{(BAP):}\quad & \underset{x\in X}{\min} & \|x-x_{0}\|\label{eq:Proj-pblm}\\
 & \mbox{s.t. } & x\in K:=\bigcap_{i=1}^{r}K_{i}.\nonumber 
\end{eqnarray}
The case when the BAP is infeasible was discussed in \cite[Section 6]{cut_Pang12}.
In this section, we discuss the case where the SIP and CIP are infeasible,
and show that for the SIP and BAP, one can use a finite number of
operations to find a certificate of infeasibility. We also make another
observation for the BAP.

We present the algorithm for the BAP needed for future discussions.
\begin{algorithm}
\label{alg:BAP}(BAP algorithm) For a point $x_{0}$ and closed convex
sets $K_{l}$, $l=1,2,\dots,r$, of a Hilbert space $X$, find the
closest point to $x_{0}$ in $K:=\cap_{l=1}^{r}K_{l}$. 

\textbf{\uline{Step 0:}}\textbf{ }Let $i=0$.

\textbf{\uline{Step 1:}} For $l\in\{1,\dots,r\}$, define $x_{i}^{(l)}\in X$,
$a_{i}^{(l)}\in X$ and $b_{i}^{(l)}\in\mathbb{R}$ by 
\begin{eqnarray*}
x_{i}^{(l)} & = & P_{K_{l}}(x_{i}),\\
a_{i}^{(l)} & = & x_{i}-x_{i}^{(l)},\\
\mbox{and }b_{i}^{(l)} & = & \left\langle a_{i}^{(l)},x_{i}^{(l)}\right\rangle .
\end{eqnarray*}
For each $i$, we let $l_{i}^{*}\in\{1,\dots,r\}$ be such that 
\[
l_{i}^{*}:=\arg\max_{1\leq l\leq r}\|x_{i}-P_{K_{l}}(x_{i})\|.
\]
\textbf{\uline{Step 2:}} Choose $S_{i}\subset\{\max(i-\bar{p},0),\dots,i\}\times\{1,\dots,r\}$,
and define $F_{i}\subset X$ by
\begin{eqnarray}
F_{i} & := & \bigcap_{(j,l)\in S_{i}}H_{(j,l)}^{\leq},\label{eq:dykstra-all-planes}\\
\mbox{where }H_{j,i}^{\leq} & := & \left\{ x:\left\langle a_{j}^{(l)},x\right\rangle \leq b_{j}^{(l)}\right\} .
\end{eqnarray}
Let $x_{i+1}=P_{F_{i}}(x_{0})$.

\textbf{\uline{Step 3:}} Set $i\leftarrow i+1$, and go back to
step 1. 
\end{algorithm}
We now present our first result that will supersede \cite[Theorem 6.1]{cut_Pang12}. 
\begin{thm}
\label{thm:No-cluster-pt}(No strong cluster points under infeasibility)
Let $X$ be a Hilbert space, $x_{0}\in\mathbb{R}^{n}$, and let $K_{l}\subset X$
be closed convex sets for $l=1,\dots,r$. Suppose $K:=\cap_{l=1}^{r}K_{l}$
is the empty set. Consider the following scenarios for solving the
SIP, CIP and BAP respectively.
\begin{enumerate}
\item Suppose Algorithm \ref{alg:Mass-proj-alg} is run for the SIP with
$\bar{p}=\infty$ at step 0 and the conditions \begin{subequations}
\label{enu:nec-condn-no-cluster-pts} 
\begin{eqnarray}
S_{i} & \subset & S_{i+1}\mbox{ for all }i\geq0,\label{eq:S-subset-S}\\
\mbox{ and }(i,l_{i}^{*}) & \in & S_{i\phantom{+1}}\mbox{ for all }i\geq0\label{eq:take-best-halfsp}
\end{eqnarray}
\end{subequations}hold. If $\tilde{F}_{i}=\emptyset$ for some $i$,
then infeasibility is detected. Otherwise, the sequence $\{x_{i}\}$
produced cannot have strong cluster points.
\item Suppose Algorithm \ref{thm:SHQP-alg} is run for the CIP with $\bar{p}=\infty$
in step 0, and $S_{i}=\{0,1,\dots,i\}$ for all $i\geq0$. If $F_{i}=\emptyset$
for some $i$, then infeasibility is detected. Otherwise, the sequence
$\{x_{i}\}$ produced cannot have strong cluster points.
\item Suppose Algorithm \ref{alg:BAP} is run for the BAP, and \eqref{enu:nec-condn-no-cluster-pts}
holds. If $F_{i}=\emptyset$ for some $i$, then infeasibility is
detected. Otherwise, the sequence $\{x_{i}\}$ produced cannot have
strong cluster points.
\end{enumerate}
\end{thm}
Note that the condition \eqref{eq:S-subset-S} implies that the feasible
sets $F_{i}$ are such that $F_{i+1}\subset F_{i}$ (or $\tilde{F}_{i+1}\subset\tilde{F}_{i}$
for Algorithm \ref{alg:Mass-proj-alg}) for all $i\geq0$. Condition
\eqref{eq:take-best-halfsp} implies that a halfspace generated by
projecting $x_{i}$ onto a set furthest away from $x_{i}$ is taken.
We continue with the proof of Theorem \ref{thm:No-cluster-pt}.
\begin{proof}
We consider (1) first. Suppose on the contrary that $\{x_{i}\}$ has
a strong cluster point, say $\bar{x}$. Assume without loss of generality
that $K_{1}$ is such that $d(\bar{x},K_{1})=\max_{l\in\{1,\dots,r\}}d(\bar{x},K_{l})$.
There must be a point, say $x_{i}$, such that $\|x_{i}-\bar{x}\|<\frac{1}{3}d(\bar{x},K_{1})$.
It is elementary to see that $d(x_{i},K_{1})>2\|x_{i}-\bar{x}\|$.
So the distance from $x_{i}$ to a halfspace produced by projecting
$x_{i}$ onto $K_{1}$ is at least $2\|x_{i}-\bar{x}\|$, and this
halfspace must not contain $\bar{x}$. This implies that $\bar{x}\notin\tilde{F}_{i}$,
and that $\bar{x}\notin\tilde{F}_{j}$ for all $j\geq i$. Therefore
$\bar{x}$ cannot be a strong cluster point.

The proof of (3) is the exactly the same as that for (1).

The proof of (2) is similar with slightly different constants. Suppose
on the contrary that $\{x_{i}\}$ has a strong cluster point $\bar{x}$,
where $f(\bar{x})>0$. Note that the distance from $x_{i}$ to the
halfspace 
\[
\{x:f(x_{i})+\langle x-x_{i},s_{i}\rangle\leq0\}
\]
 is equal to $\frac{1}{\|s_{i}\|}f(x_{i})$. In view of the outer
semicontinuity of the subdifferential mapping for convex functions,
there is a neighborhood $U$ of $\bar{x}$ such that $\|s\|<\max_{y\in\partial f(\bar{x})}\|y\|+1$
for all $x\in U$ and $s\in\partial f(x)$. Also, continuity of $f(\cdot)$
implies that $f(x_{i})$ can be arbitrarily close to $f(\bar{x})$.
If $x_{i}$ is close enough to $\bar{x}$ such that 
\[
\|x_{i}-\bar{x}\|<\frac{1}{2\|s_{i}\|}f(x_{i}),
\]
then we will get a similar contradiction as before.
\end{proof}
We improve the techniques in \cite[Theorem 6.2]{cut_Pang12} to prove
the following result. 
\begin{thm}
\label{thm:Recession-dirns}(Recession directions) Let $x_{0}\in\mathbb{R}^{n}$,
and let $K_{l}\subset\mathbb{R}^{n}$ be closed convex sets such that
$K:=\cap_{l=1}^{r}K_{l}=\emptyset$. Suppose either 
\begin{itemize}
\item Algorithm \ref{alg:Mass-proj-alg} is run with $\bar{p}=\infty$ in
step 0 and the conditions \eqref{enu:nec-condn-no-cluster-pts} hold,
or 
\item Algorithm \ref{alg:BAP} is run so that the conditions \eqref{enu:nec-condn-no-cluster-pts}
hold.
\end{itemize}
Let $\{x_{i}\}$ be the sequence produced. Then any cluster point
of $\{\frac{x_{i}}{\|x_{i}\|}\}_{i}$ lies in $R(K_{l})$, the recession
cone of $K_{l}$, for all $l\in\{1,\dots,r\}$.\end{thm}
\begin{proof}
The proof here is improved from that of \cite[Theorem 6.2]{cut_Pang12}.
Let $\{\frac{\tilde{x}_{i}}{\|\tilde{x}_{i}\|}\}$ be a subsequence
of $\{\frac{x_{i}}{\|x_{i}\|}\}$ such that 
\begin{equation}
\lim_{i\to\infty}\|\tilde{x}_{i}\|=\infty\mbox{ and }\lim_{i\to\infty}\frac{\tilde{x}_{i}}{\|\tilde{x}_{i}\|}=v.\label{eq:inf-and-lim}
\end{equation}
We show that such a limit has to lie in $R(K_{l})$ for all $l\in\{1,\dots,r\}$.
Seeking a contradiction, suppose that $v$ is such that $v\notin R(K_{l})$
for some $l\in\{1,\dots,r\}$. Let $L_{1}\subset\{1,\dots,r\}$ be
such that $v\in R(K_{l})$ for all $l\in L_{1}$, and $L_{2}=\{1,\dots,r\}\backslash L_{1}$. 

\textbf{Claim 1:} For each $l\in L_{2}$, there is a unit vector $w_{l}\in\mathbb{R}^{n}$
and $M_{l}\in\mathbb{R}$ such that $\langle w_{l},v\rangle>0$, and
$\langle w_{l},c\rangle\leq M_{l}$ for all $c\in K_{l}$.

Take any point $y_{l}\in K_{l}$. Since $v\notin R(K_{l})$, there
is some $\gamma\geq0$ such that $y_{l}+\gamma v\in K_{l}$, but $y_{l}+\gamma^{\prime}v\notin K_{l}$
for all $\gamma^{\prime}>\gamma$. It follows that there exists a
unit vector $w_{l}\in N_{K_{l}}(y_{l}+\gamma v)$ such that $\langle w_{l},v\rangle>0$,
and we can take $M_{l}=\langle w_{l},y_{l}+\gamma v\rangle$. This
ends the proof of Claim 1.

\textbf{Claim 2:} There is an $I$ such that if $i>I$, $\arg\max_{l\in\{1,\dots,r\}}d(\tilde{x}_{i},K_{l})\in L_{2}$. 

In view of \eqref{eq:inf-and-lim}, we can write $\tilde{x}_{i}=\rho_{i}[v+\beta_{i}]+\tilde{x}_{0}$,
where $\rho_{i}\in\mathbb{R}$ and $\beta_{i}\in\mathbb{R}^{n}$ are
such that $\rho_{i}\to\infty$ and $\beta_{i}\to0$. For each $l\in L_{1}$,
we can choose $y_{l}\in K_{l}$ and have $y_{l}+\mathbb{R}_{+}\{v\}\subset K_{l}$.
We can then choose $I$ large enough so that $P_{y_{l}+\mathbb{R}_{+}(v)}(\tilde{x}_{i})=P_{y_{l}+\mathbb{R}(v)}(\tilde{x}_{i})$
for all $i>I$ and $l\in L_{1}$.

We now estimate $d(\tilde{x}_{i},K_{l})$ for $l\in L_{1}$ and $i$
large enough. 
\begin{eqnarray}
d(\tilde{x}_{i},K_{l}) & \leq & d(\tilde{x}_{i},y_{l}+\mathbb{R}(v))\label{eq:est-L1}\\
 & = & d(\rho_{i}[v+\beta_{i}]+\tilde{x}_{0},y_{l}+\mathbb{R}(v))\nonumber \\
 & = & d(\rho_{i}\beta_{i}+\tilde{x}_{0}-y_{l},\mathbb{R}(v))\nonumber \\
 & \leq & d(\rho_{i}\beta_{i}+\tilde{x}_{0}-y_{l},0)\nonumber \\
 & \leq & \rho_{i}\|\beta_{i}\|+\|\tilde{x}_{0}-y_{l}\|.\nonumber 
\end{eqnarray}
For $l\in L_{2}$, let the halfspace $H_{l}$ be $\{x:\langle w_{l},x\rangle\leq M_{l}\}$.
Claim 1 says that we have $K_{l}\subset H_{l}$ for some $w_{l}\in\mathbb{R}^{n}$
with $\|w_{l}\|=1$ such that $\langle w_{l},v\rangle>0$. For $i$
large enough, we have 
\begin{eqnarray}
d(\tilde{x}_{i},K_{l}) & \geq & d(\tilde{x}_{i},H_{l})\label{eq:est-L2}\\
 & = & d(\rho_{i}[v+\beta_{i}]+\tilde{x}_{0},H_{l})\nonumber \\
 & = & \left\langle w_{l},\rho_{i}[v+\beta_{i}]+\tilde{x}_{0}\right\rangle -M_{l}\nonumber \\
 & = & \rho_{i}\langle w_{l},v\rangle+\rho_{i}\langle w_{l},\beta_{i}\rangle+\langle w_{l},\tilde{x}_{0}\rangle-M_{l}\nonumber \\
 & \geq & \rho_{i}\langle w_{l},v\rangle-\rho_{i}\|\beta_{i}\|+\langle w_{l},\tilde{x}_{0}\rangle-M_{l}.\nonumber 
\end{eqnarray}
From \eqref{eq:est-L1} and \eqref{eq:est-L2}, we can estimate that
$\arg\max_{l\in\{1,\dots,r\}}d(\tilde{x}_{i},K_{l})\in L_{2}$ for
all $i$ large enough as needed. This ends the proof of Claim 2.

Let $c_{i}:=P_{K_{l_{i}^{*}}}(\tilde{x}_{i})$, and let $u_{i}$ be
the unit vector in the direction of $\tilde{x}_{i}-c_{i}$. We write
$\tilde{x}_{i}-c_{i}=\alpha_{i}u_{i}$. We have 
\begin{equation}
\langle u_{i},c_{i}\rangle=\langle u_{i},\tilde{x}_{i}-\alpha_{i}u_{i}\rangle=\langle u_{i},\tilde{x}_{i}\rangle-\alpha_{i}.\label{eq:uc-uxa}
\end{equation}
If $i$ is large enough, there is a tail of the sequence $\{l_{i}^{*}\}$
which lies in $L_{2}$. Let $l^{*}$ be an index in $L_{2}$ that
appears in the sequence $\{l_{i}^{*}\}$ infinitely often. We let
$w:=w_{l^{*}}$ and $M=M_{l^{*}}$. So
\[
\alpha_{i}\langle w,u_{i}\rangle=\langle w,\tilde{x}_{i}\rangle-\langle w,c_{i}\rangle\geq\langle w,\tilde{x}_{i}\rangle-M.
\]
When $i$ is large enough, we have $\alpha_{i}\langle w,u_{i}\rangle=\langle w,\tilde{x}_{i}-c_{i}\rangle>M-M=0$,
so $\langle w,u_{i}\rangle>0$. Hence $\alpha_{i}\geq\alpha_{i}\langle w,u_{i}\rangle\geq\langle w,\tilde{x}_{i}\rangle-M$.
Therefore, making use of \eqref{eq:uc-uxa}, we have 
\[
\langle u_{i},c_{i}\rangle\leq\langle u_{i},\tilde{x}_{i}\rangle-\langle w,\tilde{x}_{i}\rangle+M.
\]
By the workings of the corresponding algorithms, we have $\langle u_{i},\tilde{x}_{i}\rangle>\langle u_{i},c_{i}\rangle$
and $\langle u_{i},\tilde{x}_{j}\rangle\leq\langle u_{i},c_{i}\rangle$
for all $j>i$. This gives $\langle u_{i},\tilde{x}_{j}-\tilde{x}_{i}\rangle\leq0$,
which gives $\langle u_{i},v\rangle\leq0$. 

Let $u$ be a cluster point of $\{u_{i}\}$. We can consider subsequences
so that $l_{i}^{*}=l^{*}$ and $\lim_{i\to\infty}u_{i}$ exists. For
any point $c\in K_{l^{*}}$, we have 
\begin{eqnarray*}
\langle u,c\rangle & = & \lim_{i\to\infty}\langle u_{i},c\rangle\\
 & \leq & \liminf_{i\to\infty}\langle u_{i},c_{i}\rangle\\
 & \leq & \liminf_{i\to\infty}[\langle u_{i},\tilde{x}_{i}\rangle-\langle w,\tilde{x}_{i}\rangle+M]\\
 & = & \liminf_{i\to\infty}\|\tilde{x}_{i}\|\left(\left\langle u_{i},\frac{\tilde{x}_{i}}{\|\tilde{x}_{i}\|}\right\rangle -\left\langle w,\frac{\tilde{x}_{i}}{\|\tilde{x}_{i}\|}\right\rangle \right)+M\\
 & = & \liminf_{i\to\infty}\|\tilde{x}_{i}\|[\langle u_{i},v\rangle-\langle w,v\rangle]+M\\
 & = & -\infty,
\end{eqnarray*}
which is absurd. The contradiction gives $v\in R(K_{l})$ for all
$l\in\{1,\dots,r\}$.
\end{proof}
By combining Theorems \ref{thm:No-cluster-pt} and \ref{thm:Recession-dirns},
we can conclude the following.
\begin{cor}
(Certifying infeasibility in finitely many operations) Suppose $X=\mathbb{R}^{n}$,
$K=\emptyset$ and $\cap_{l=1}^{r}R(K_{l})=\emptyset$ in either Algorithm
\ref{alg:Mass-proj-alg} run with $\bar{p}=\infty$ in step 0 or Algorithm
\ref{alg:BAP}. Suppose also that the condition \eqref{enu:nec-condn-no-cluster-pts}
is satisfied. For any starting point $x_{0}$, the algorithms terminate
with $\tilde{F}_{i}=\emptyset$ or $F_{i}=\emptyset$ after finitely
many iterations.
\end{cor}
Examples involving 
\begin{eqnarray*}
K_{1} & = & \{(x,y)\in\mathbb{R}^{2}\mid y\geq e^{-x}\},\\
\mbox{ and }K_{2} & = & \{(x,y)\in\mathbb{R}^{2}\mid y\leq-e^{-x}\}
\end{eqnarray*}
show that when $X=\mathbb{R}^{n}$, $K=\emptyset$ and $\cap_{l=1}^{r}R(K_{l})\neq\emptyset$,
one may not get a certificate of infeasibility in finitely many operations.
\begin{rem}
(Infeasibility certificate) Note that when $X=\mathbb{R}^{n}$, the
normals of the halfspaces need to be linearly dependent before infeasibility
can be detected. Since a set of $n$ vectors in $\mathbb{R}^{n}$
can be arbitrarily close to another set of $n$ linearly independent
vectors, one might need a large number of halfspaces to detect infeasibility.
The QP algorithm in \cite{Goldfarb_Idnani} for example gives a (Farkas
Lemma type) certificate of infeasibility of a system $Ax\leq b$,
where $A\in\mathbb{R}^{m\times n}$ and $b\in\mathbb{R}^{m}$, by
finding a vector $r\in\mathbb{R}^{n}$ such that 
\begin{equation}
r\in\mathbb{R}_{+}^{n},\, r^{T}A=0\mbox{ and }r^{T}b<0.\label{eq:Farkas-type}
\end{equation}
An acceptable relaxation of $r^{T}A=0$ would be that $r^{T}A$ being
approximately zero.
\end{rem}
We remark on aggregation of constraints.
\begin{rem}
\label{rem:aggregation}(Aggregation of constraints) Observe that,
in Algorithm \ref{alg:BAP}, $x_{i}$ is the projection of $x_{0}$
onto $F_{i-1}$. We can aggregate some of the constraints describing
$F_{i-1}$ in a manner similar to bundle methods. More precisely,
for constraints $a_{j}^{T}x\leq b_{j}$ for all $j\in J$, we can
find multipliers $\lambda_{j}\geq0$ for all $j\in J$ such that we
now store the single constraint 
\[
\bigg[\sum_{j\in J}\lambda_{j}a_{j}\bigg]^{T}x\leq\bigg[\sum_{j\in J}\lambda_{j}b_{j}\bigg].
\]
The polyhedron thus produced would be larger than $F_{i-1}$, but
fewer linear constraints need to be stored. We only require that $x_{i}$
is the projection of $x_{0}$ onto the new polyhedron. Projecting
$x_{i}$ onto a set $K_{l}$ not containing $x_{i}$ produces a new
polyhedron, from which we get $\|x_{i+1}-x_{0}\|>\|x_{i}-x_{0}\|$. 
\end{rem}
We present another result for detecting infeasibility in the BAP.
\begin{thm}
\label{thm:BAP-to-infinity}(BAP) Let $X$ be a Hilbert space, $x_{0}\in X$
and $K_{l}\subset X$ be closed convex sets for $l\in\{1,\dots,r\}$.
Suppose Algorithm \ref{alg:BAP} is modified such that 
\begin{enumerate}
\item The aggregation of constraints mentioned in Remark \ref{rem:aggregation}
is carried out, and that $x_{i}$ is the still the projection of $x_{0}$
onto the new polyhedron.
\item Halfspaces obtained by projecting $x_{i}$ onto some of the $K_{l}$
are intersected with the polyhedron stored earlier. One of these halfspaces
is obtained by projecting $x_{i}$ onto the set $K_{l_{i}^{*}}$ for
which $d(x_{i},K_{l_{i}^{*}})=\max_{l\in\{1,\dots,r\}}d(x_{i},K_{l})$.
The next iterate $x_{i+1}$ is the projection of $x_{0}$ onto this
new polyhedron. 
\end{enumerate}
Then the sequence of iterates $\{x_{i}\}$ cannot have a strong cluster
point. If $X=\mathbb{R}^{n}$, the sequence $\{\|x_{i}-x_{0}\|\}_{i}$
must be monotonically increasing to infinity.\end{thm}
\begin{proof}
Suppose on the contrary that $\bar{x}$ is a strong cluster point.
Without loss of generality, let $K_{1}$ be such that $d(\bar{x},K_{1})=\max_{l\in\{1,\dots,r\}}d(\bar{x},K_{l})$.
If $\|x_{i}-x_{0}\|>\|\bar{x}-x_{0}\|$, then $\|x_{j}-x_{0}\|>\|\bar{x}-x_{0}\|$
for all $j\geq i$, which means that $\bar{x}$ cannot be a cluster
point. Suppose $x_{i}$ is close enough to $\bar{x}$ so that 
\[
\|x_{i}-x_{0}\|^{2}+[d(\bar{x},K_{1})-\|x_{i}-\bar{x}\|]^{2}>\|\bar{x}-x_{0}\|^{2}.
\]
Then $d(x_{i},K_{1})\geq d(\bar{x},K_{1})-\|x_{i}-\bar{x}\|$. In
other words, the distance of $x_{i}$ to the halfspace produced by
projecting $x_{i}$ onto $K_{1}$ is at least $d(\bar{x},K_{1})-\|x_{i}-\bar{x}\|$. 

In view of (1), we have $\angle x_{i+1}x_{i}x_{0}\geq\pi/2$. It is
elementary to check that 
\[
\|x_{i+1}-x_{0}\|^{2}\geq\|x_{i}-x_{0}\|^{2}+[d(\bar{x},K_{1})-\|x_{i}-\bar{x}\|]^{2}>\|\bar{x}-x_{0}\|^{2},
\]
which would imply that $\bar{x}$ cannot be a cluster point of $\{x_{i}\}$.
The second sentence of the theorem is easy.
\end{proof}

\section{\label{sec:inner-GI}Practical implementation using a dual QP algorithm}

The algorithms in this paper have been shown to enjoy several useful
theoretical properties mentioned in earlier parts of this paper. However,
the size of the QPs that need to be solved may be too large in practice
for a practical QP solver. In this section, we explain that the dual
active set QP algorithm of Goldfarb and Idnani \cite{Goldfarb_Idnani}
can give useful iterates even when the QPs are not solved to optimality.

For the sake of simplicity, we repeat the narrative of \cite{Goldfarb_Idnani}
where only inequality constraints are considered. Let $C\in\mathbb{R}^{n\times m}$
and $b\in\mathbb{R}^{m}$ and consider the QP that arises repeatedly
in our algorithms
\begin{eqnarray*}
QP(y,C,b):= & \underset{\tilde{x}\in\mathbb{R}^{n}}{\min} & \frac{1}{2}\|\tilde{x}-y\|^{2}\\
 & \mbox{s.t.} & C^{T}\tilde{x}\leq b.
\end{eqnarray*}
In other words, we want to project the point $y$ onto the polyhedron
$F:=\{\tilde{x}:C^{T}\tilde{x}\leq b\}$. A positive definite Hessian
in the QP is required for the algorithm in \cite{Goldfarb_Idnani}
to work, which is indeed met in our context, where the Hessian of
our QP is the identity matrix. At the projection $P_{F}(y)$, not
every constraint in $C^{T}\tilde{x}\leq b$ is tight. 

The dual QP algorithm of \cite{Goldfarb_Idnani} starts with a candidate
active index set $A_{0}:=\emptyset$ and $\tilde{x}_{0}:=y$. Note
that $\tilde{x}_{i}=P_{F_{A_{i}}}(y)$ for $i=0$, where 
\[
F_{A_{i}}:=\{\tilde{x}:c_{j}^{T}\tilde{x}\leq b_{j}\mbox{ for all }j\in A_{i}\},
\]
with $c_{j}$ being the $j$th column of $C$. The structure $\tilde{x}_{i}=P_{F_{A_{i}}}(y)$
and $c_{j}^{T}\tilde{x}_{i}=b_{j}$ for all $j\in A_{i}$ would be
maintained throughout the algorithm over all iterations $i$ until
termination at an optimal active set $\bar{A}$, where $P_{F_{\bar{A}}}(y)=P_{F}(y)$.
In each iteration $i$, the next index set $A_{i}$ is determined
by first choosing some $j\notin A_{i-1}$ such that $c_{j}^{T}\tilde{x}_{i}>b_{j}$.
Next, the active set $A_{i}$ is updated so that $A_{i}\subset A_{i-1}\cup\{j\}$.
Some useful properties are:
\begin{enumerate}
\item $\|\tilde{x}_{i}-y\|$ is monotonically increasing, and
\item The dual QP algorithm converges in finitely many iterations to $P_{F}(y)$. 
\end{enumerate}
Suppose at the $i$th iteration of the algorithms in earlier sections,
we want to project $x_{i}$ onto some polyhedron $F$. We would let
$\tilde{x}_{0}:=x_{i}$, and run the GI algorithm to get a sequence
of iterates $\{\tilde{x}_{j}\}_{j}$ that converges to $P_{F}(\tilde{x}_{0})=P_{F}(x_{i})$
in finitely many steps by property (2). This convergence property
is reassuring, but the number of iterations may still be prohibitively
large. We now show that the iterates $\{\tilde{x}_{j}\}_{j}$ get
better for the associated feasibility problem, even if we don't arrive
at $P_{F}(\tilde{x}_{0})$. Since $\tilde{x}_{j}=P_{F_{A_{j}}}(\tilde{x}_{0})$
and $F\subset F_{A_{j}}$, Fej\'{e}r monotonicity \eqref{eq:Fejer-mon-eq}
implies that 
\begin{equation}
\|\tilde{x}_{j}-c\|^{2}\leq\|x_{i}-c\|^{2}-\|x_{i}-\tilde{x}_{j}\|^{2}\mbox{ for all }c\in F.\label{eq:Fejer-GI}
\end{equation}
Property (1) implies that the term $\|x_{i}-\tilde{x}_{j}\|^{2}$
increases with the number of iterations $j$ in the GI algorithm,
so \eqref{eq:Fejer-GI} implies that $\|\tilde{x}_{j}-c\|^{2}$ decreases
for all $c\in F$, and by at least the factor $\|x_{i}-\tilde{x}_{j}\|^{2}$.
In other words, the iterates $\tilde{x}_{j}$ get better for the associated
SIP, CIP or BAP. For the SIP and CIP, we may arrive at a point satisfying
the conditions in step 2 of Algorithm \ref{alg:Mass-proj-alg} or
\ref{thm:SHQP-alg} without solving the QP to optimality by considering
$x_{i}+t(\tilde{x}_{j}-x_{i})$ for some $t\in[1,2]$. Recall that
our earlier results tell us that such a point can still give multiple-term
superlinear convergence. 

We call the inner iterations to solve the QPs \emph{inner GI steps}.
The inner GI steps in the dual QP algorithm allows for the underlying
QP in the BAP to be solved using warmstart solutions from previous
iterations. In other words, the QPs do not have to be solved from
scratch.

Note that for the BAP, we project from $x_{0}$ all the time, but
in the SIP, we project from $x_{i}$ at the $i$th iteration, and
$x_{i}$ is a point closer to $K:=\cap_{l=1}^{r}K_{l}$ than $x_{0}$:
We think this is a reason why the SIP is easier to solve than the
BAP. 

In prevailing SIP algorithms, the operations that can be taken are
(1) to find supporting halfspaces of $K_{l}$ by projecting from $x_{i}$,
and (2) to move $x_{i}$ to a point $x_{i+1}$ by various strategies.
We propose a new operation: (3) to perform inner GI steps to find
better candidates for $x_{i+1}$ before performing operation (2).
 By introducing operation (3), we can reduce the SIP for sets with
smooth boundaries to Newton-like methods that give superlinear convergence
in the manner of \cite{G-P98,G-P01} or of the algorithms in this
paper, and such fast convergence had indeed been observed. Further
details are discussed in \cite{SHDQP}.

\section{Conclusion}

We have done what we set out to do in Subsection \ref{sub:Contrib}.
The SIP and the CIP can be cast as a problem of finding an $x$ such
that 
\[
\max_{l\in\{1,\dots,r\}}f_{l}(\cdot)\leq0
\]
 or to give a certificate of nonexistence if no such $x$ exists,
where each $f_{l}(\cdot)$ is convex. Our algorithms here can achieve
multiple-term superlinear or multiple-term quadratic convergence if
the proper conditions hold.  The BAP \eqref{eq:Proj-pblm} cannot
be written in this form, which may be why we cannot expect the fast
convergence for the BAP in general. We also make further observations
on the infeasible case in Section \ref{sec:Infeasible}, showing that
under reasonable conditions, a finite number of operations can give
a certificate of infeasibility for both the SIP and BAP.

\bibliographystyle{amsalpha}
\bibliography{../refs}

\newcommand{\etalchar}[1]{$^{#1}$}
\providecommand{\bysame}{\leavevmode\hbox to3em{\hrulefill}\thinspace}
\providecommand{\MR}{\relax\ifhmode\unskip\space\fi MR }
\providecommand{\MRhref}[2]{%
  \href{http://www.ams.org/mathscinet-getitem?mr=#1}{#2}
}
\providecommand{\href}[2]{#2}
\begin{thebibliography}{CCC{\etalchar{+}}12}

\bibitem[Agm83]{Agmon54}
S.~Agmon, \emph{The relaxation method for linear inequalities}, Canad. J. Math.
  \textbf{4} (1983), 479--489.

\bibitem[BB96]{BB96_survey}
H.H. Bauschke and J.M. Borwein, \emph{On projection algorithms for solving
  convex feasibility problems}, SIAM Rev. \textbf{38} (1996), 367--426.

\bibitem[BCK06]{BausCombKruk06}
H.H. Bauschke, P.L. Combettes, and S.G. Kruk, \emph{Extrapolation algorithm for
  affine-convex feasibility problems}, Numer. Algorithms \textbf{41} (2006),
  239--274.

\bibitem[BD85]{BD86}
J.P. Boyle and R.L. Dykstra, \emph{A method for finding projections onto the
  intersection of convex sets in {H}ilbert spaces}, Advances in Order
  Restricted Statistical Inference, Lecture notes in Statistics, Springer, New
  York, 1985, pp.~28--47.

\bibitem[BDHP03]{BDHP03}
H.H. Bauschke, F.~Deutsch, H.S. Hundal, and S.-H. Park, \emph{Accelerating the
  convergence of the method of alternating projections}, Trans. Amer. Math.
  Soc. \textbf{355} (2003), no.~9, 3433--3461.

\bibitem[BR09]{BR09}
E.G. Birgin and M.~Raydan, \emph{{D}ykstra's algorithm and robust stopping
  criteria}, Encyclopedia of Optimization (C.~A. Floudas and P.~M. Pardalos,
  eds.), Springer, US, 2 ed., 2009, pp.~828--833.

\bibitem[BZ05]{BZ05}
J.M. Borwein and Q.J. Zhu, \emph{Techniques of variational analysis}, Springer,
  NY, 2005, CMS Books in Mathematics.

\bibitem[CCC{\etalchar{+}}12]{CensorChenCombettesDavidiHerman12}
Y.~Censor, W.~Chen, P.~L. Combettes, R.~Davidi, and G.T. Herman, \emph{On the
  effectiveness of projection methods for convex feasibility problems with
  linear inequality constraints}, Comput. Optim. Appl. \textbf{51} (2012),
  1065--1088.

\bibitem[Cen84]{Censor84}
Y.~Censor, \emph{Iterative methods for the convex feasibility problem}, Ann.
  Discrete Math. \textbf{20} (1984), 83--91.

\bibitem[Cim38]{Cimmino38}
G.~Cimmino, \emph{Calcolo approssimato per le soluzioni dei sistemi di
  equazioni lineari}, La Ricerca Scientifica XVI \textbf{II} (1938), no.~9,
  326--333.

\bibitem[Cla83]{Cla83}
F.H. Clarke, \emph{Optimization and nonsmooth analysis}, Wiley, Philadelphia,
  1983, Republished as a SIAM Classic in Applied Mathematics, 1990.

\bibitem[Com93]{Combettes93}
P.L. Combettes, \emph{The foundations of set theoretic estimation}, Proc. IEEE
  \textbf{81} (1993), 182--208.

\bibitem[Com96]{Combettes96}
\bysame, \emph{The convex feasibility problem in image recovery}, Advances in
  Imaging and Electron Physics (P.~Hawkes, ed.), vol.~95, Academic, New York,
  1996, pp.~155--270.

\bibitem[CZ97]{CensorZenios97}
Y.~Censor and S.A. Zenios, \emph{Parallel optimization}, Oxford University
  Press, 1997.

\bibitem[Deu95]{Deutsch95}
F.~Deutsch, \emph{The angle between subspaces of a {H}ilbert space},
  Approximation Theory, Spline Functions and Applications (S.~P. Singh, ed.),
  Kluwer Academic Publ., The Netherlands, 1995, pp.~107--130.

\bibitem[Deu01a]{Deutsch01_survey}
\bysame, \emph{Accelerating the convergence of the method of alternating
  projections via a line search: A brief survey}, Inherently Parallel
  Algorithms in Feasibility and Optimization and their Applications
  (D.~Butnariu, Y.~Censor, and S.~Reich, eds.), Elsevier, 2001, pp.~203--217.

\bibitem[Deu01b]{Deustch01}
\bysame, \emph{Best approximation in inner product spaces}, Springer, 2001, CMS
  Books in Mathematics.

\bibitem[Dyk83]{Dykstra83}
R.L. Dykstra, \emph{An algorithm for restricted least-squares regression}, J.
  Amer. Statist. Assoc. \textbf{78} (1983), 837--842.

\bibitem[ER11]{EsRa11}
R.~Escalante and M.~Raydan, \emph{Alternating projection methods}, SIAM, 2011.

\bibitem[FL03]{FletcherLeyffer03}
R.~Fletcher and S.~Leyffer, \emph{Filter-type algorithms for solving systems of
  algebraic equations and inequalities}, High Performance Algorithms and
  Software for Nonlinear Optimization (G.~di~Pillo and A.~Murli, eds.), Kluwer,
  2003, pp.~265--284.

\bibitem[Fuk82]{Fukushima82}
M.~Fukushima, \emph{A finitely convergent algorithm for convex inequalities},
  IEEE Trans. Automat. Control \textbf{27} (1982), no.~5, 1126--1127.

\bibitem[GI83]{Goldfarb_Idnani}
D.~Goldfarb and A.~Idnani, \emph{A numerically stable dual method for solving
  strictly convex quadratic programs}, Math. Programming \textbf{27} (1983),
  1--33.

\bibitem[GK89]{GK89}
W.B. Gearhart and M.~Koshy, \emph{Acceleration schemes for the method of
  alternating projections}, J. Comput. Appl. Math. \textbf{26} (1989),
  235--249.

\bibitem[GP98]{G-P98}
U.M. Garc{\'i}a-Palomares, \emph{A superlinearly convergent projection
  algorithm for solving the convex inequality problem}, Oper. Res. Lett.
  \textbf{22} (1998), 97--103.

\bibitem[GP01]{G-P01}
\bysame, \emph{Superlinear rate of convergence and optimal acceleration schemes
  in the solution of convex inequality problems}, Inherently Parallel
  Algorithms in Feasibility and Optimization and their Applications
  (D.~Butnariu, Y.~Censor, and S.~Reich, eds.), Elsevier, 2001, pp.~297--305.

\bibitem[GPR67]{GPR67}
L.G. Gubin, B.T. Polyak, and E.V. Raik, \emph{The method of projections for
  finding the common point of convex sets}, USSR Comput. Math. Math. Phys.
  \textbf{7} (1967), no.~6, 1--24.

\bibitem[Han88]{Han88}
S.P. Han, \emph{A successive projection method}, Math. Programming \textbf{40}
  (1988), 1--14.

\bibitem[HC08]{HermanChen08}
G.T. Herman and W.~Chen, \emph{A fast algorithm for solving a linear
  feasibility problem with application to intensity-modulated radiation
  therapy}, Linear Algebra Appl. \textbf{428} (2008), 1207--1217.

\bibitem[HUL93]{HiriartUrrutyLamerechal93a}
J.-B. Hiriart-Urruty and C.~Lemar{\'e}chal, \emph{Convex analysis and
  minimization algorithms {I} \& {II}}, Springer, 1993, Grundlehren der
  mathematischen Wissenschaften, Vols 305 \& 306.

\bibitem[Kiw95]{Kiwiel95}
K.C. Kiwiel, \emph{Block-iterative surrogate projection methods for convex
  feasibility problems}, Linear Algebra Appl. \textbf{215} (1995), 225--259.

\bibitem[LLM09]{LLM09_lin_conv_alt_proj}
A.S. Lewis, D.R. Luke, and J.~Malick, \emph{Local linear convergence for
  alternating and averaged nonconvex projections}, Found. Comput. Math.
  \textbf{9} (2009), no.~4, 485--513.

\bibitem[LM08]{LewisMalick08}
A.S. Lewis and J.~Malick, \emph{Alternating projection on manifolds}, Math.
  Oper. Res. \textbf{33} (2008), 216--234.

\bibitem[Mif77]{Mifflin77}
R.~Mifflin, \emph{Semismooth and semiconvex functions in constrained
  optimization}, SIAM J. Control Optim. \textbf{15} (1977), 959--972.

\bibitem[Mor06]{Mor06}
B.S. Mordukhovich, \emph{Variational analysis and generalized differentiation
  {I} and {II}}, Springer, Berlin, 2006, Grundlehren der mathematischen
  Wissenschaften, Vols 330 and 331.

\bibitem[MPH81]{Mayne_Polak_Heunis81}
D.Q. Mayne, E.~Polak, and A.J. Heunis, \emph{Solving nonlinear inequalities in
  a finite number of iterations}, J. Optim. Theory Appl. \textbf{33} (1981),
  207--221.

\bibitem[Pan13]{SHDQP}
C.H.J. Pang, \emph{{SHDQP}: An algorithm for convex set intersection problems
  based on supporting hyperplanes and dual quadratic programming}, ArXiv
  e-prints (2013).

\bibitem[Pan14a]{sub_BAP}
\bysame, \emph{Accelerating the alternating projection algorithm for the case
  of affine subspaces using supporting hyperplanes}, (preprint) (2014).

\bibitem[Pan14b]{cut_Pang12}
\bysame, \emph{Set intersection problems: Supporting hyperplanes and quadratic
  programming}, Math. Programming (Online first) (2014).

\bibitem[PI88]{DePierroIusem88}
A.R.~De Pierro and A.N. Iusem, \emph{A finitely convergent {"}row-action{"}
  method for the convex feasibility problem}, Appl. Math. Optim. \textbf{17}
  (1988), 225--235.

\bibitem[Pie81]{DePierro81}
A.R.~De Pierro, \emph{Metodos de proje\c{c}{\~a}o para a resol\c{c}{\~a}o de
  sistemas gerais de equa\c{c}{\~o}es alg{\'e}bricas lienaers}, Ph.D. thesis,
  1981.

\bibitem[Pie84]{Pierra84}
G.~Pierra, \emph{Decomposition through formalization in a product space}, Math.
  Programming \textbf{28} (1984), 96--115.

\bibitem[PM79]{Polak_Mayne79}
E.~Polak and D.Q. Mayne, \emph{On the finite solution of nonlinear
  inequalities}, IEEE Trans. Automat. Control \textbf{AC-24} (1979), 443--445.

\bibitem[Rob76]{Robinson_76_CIP}
S.M. Robinson, \emph{A subgradient algorithm for solving {$K$}-convex
  inequalities}, Optimization and operations research (Proc. Conf.,
  Oberwolfach, 1975), Lecture Notes in Econom. Math. Systems, Vol. 117
  (W~Oettli and K.~Ritter, eds.), Springer, Berlin- New York, 1976,
  pp.~237--245.

\bibitem[Roc70]{Rockafellar70}
R.T. Rockafellar, \emph{Convex analysis}, Princeton, 1970.

\bibitem[RW98]{RW98}
R.T. Rockafellar and R.J.-B. Wets, \emph{Variational analysis}, Grundlehren der
  mathematischen Wissenschaften, vol. 317, Springer, Berlin, 1998.

\bibitem[San87]{DosSantos87}
L.~T.~Dos Santos, \emph{A parallel subgradient projections method for the
  convex feasibility problem}, J. Comput. Appl. Math. \textbf{18} (1987),
  307--320.

\end{thebibliography}

\end{document}